\documentclass[onecolumn]{elsart3p}

\usepackage[latin1]{inputenc}

\usepackage{mathrsfs}
\usepackage{graphicx}
\usepackage{epstopdf}
\usepackage{amsmath}
\usepackage{amssymb}
\usepackage{stmaryrd}
\usepackage{float}
\usepackage{bigints}
\usepackage{cite}
\usepackage{color}
\usepackage[abs]{overpic}
\usepackage{cases}
\usepackage{tikz}
\usepackage{algorithm,algpseudocode}
\usepackage{rotating}
\usepackage{blkarray}
\usetikzlibrary{matrix,calc,arrows}
\usepackage{soul,xcolor}
\usepackage{verbatim}
\usepackage{graphicx}
\usepackage{algorithm}
\usepackage{pifont}
\usepackage{multirow}
\usepackage{cancel}

\usepackage{bm,bbm}

\theoremstyle{remark}
\newtheorem{remark}{Remark}
\newenvironment{proof}{{\bfseries Proof.}}{}

\newcommand{\lm}[1]{{\color{black}{#1}}} 
\definecolor{darkseagreen}{rgb}{0.56, 0.74, 0.56}
\definecolor{darkspringgreen}{rgb}{0.09, 0.45, 0.27}

\newcommand{\Vertiii}[1]{{\left\vert\kern-0.25ex\left\vert\kern-0.25ex\left\vert #1 \right\vert\kern-0.25ex\right\vert\kern-0.25ex\right\vert}}

\newcommand{\E}{K}
\newcommand{\F}{F}

\renewcommand{\e}{e}
\newcommand{\taun}{\Omega_h}
\newcommand{\h}{h}
\newcommand{\hE}{\h_\E}
\newcommand{\hF}{\h_\F}
\newcommand{\he}{\h_\e}
\newcommand{\Ebold}{\mathbf E}
\newcommand{\Eboldt}{\Ebold_t}
\newcommand{\Bbold}{\mathbf B}
\newcommand{\Bboldt}{\Bbold_t}
\newcommand{\xbold}{\mathbf x}
\newcommand{\xboldE}{\xbold_\E}
\newcommand{\xboldF}{\xbold_\F}
\newcommand{\bboldF}{\mathbf b_\F}
\newcommand{\bboldE}{\mathbf b_\E}
\newcommand{\ubold}{\mathbf u}
\newcommand{\vbold}{\mathbf v}
\newcommand{\vboldI}{\vbold_I}
\newcommand{\wbold}{\mathbf w}
\newcommand{\Jbold}{\mathbf J}
\newcommand{\JboldI}{\Jbold_I}
\newcommand{\Eboldht}{\Ebold_{\h,t}}
\newcommand{\Eboldh}{\Ebold_\h}
\newcommand{\Eboldhm}{\Eboldh^m}
\newcommand{\Bboldht}{\Bbold_{\h,t}}
\newcommand{\Bboldh}{\Bbold_\h}
\newcommand{\Bboldhm}{\Bboldh^m}
\newcommand{\uboldh}{\ubold_\h}
\newcommand{\vboldh}{\vbold_\h}

\newcommand{\wboldh}{\wbold_\h}
\newcommand{\Vboldface}{\mathbf V^{\FACE}}
\newcommand{\Vboldedge}{\mathbf V^{\EDGE}}
\newcommand{\VboldfaceE}{\Vboldface(\E)}
\newcommand{\VboldedgeE}{\Vboldedge(\E)}
\newcommand{\Vboldfaceh}{\Vboldface_\h}
\newcommand{\VboldfacehE}{\Vboldfaceh(\E)}
\newcommand{\Vtildeboldfaceh}{\mathbf{\widetilde V}_\h^{\FACE}}

\newcommand{\Vtildeboldedgeh}{\mathbf{\widetilde V}_\h^{\EDGE}}

\newcommand{\Vboldedgeh}{\Vboldedge_\h}
\newcommand{\VboldedgehE}{\Vboldedgeh(\E)}
\newcommand{\VboldedgehF}{\Vboldedgeh(\F)}
\DeclareMathOperator{\curlbold}{\mathbf{curl}}
\DeclareMathOperator{\curlboldF}{\mathbf{curl}_\F}
\newcommand{\Piboldh}{\boldsymbol \Pi_\h^\text{curl}}
\newcommand{\Pizboldh}{\boldsymbol \Pi^0_\h}
\newcommand{\Pizh}{\Pi^0_\h}
\newcommand{\Pboldh}{\mathbf P_\h}
\newcommand{\sh}{s_\h}

\newcommand{\Pbb}{\mathbb P}
\newcommand{\Rbb}{\mathbb R}
\let\div\relax
\DeclareMathOperator{\div}{div}
\DeclareMathOperator{\divF}{div_\F}
\DeclareMathOperator{\rotF}{rot_\F}
\newcommand{\eboldh}{\mathbf e_\h}
\newcommand{\bboldh}{\mathbf b_\h}
\newcommand{\eboldhm}{\eboldh^m}
\newcommand{\bboldhm}{\bboldh^m}
\newcommand{\eboldhmpo}{\eboldh^{m+1}}
\newcommand{\bboldhmpo}{\bboldh^{m+1}}
\newcommand{\eboldht}{\mathbf e_{\h,t}}
\newcommand{\bboldht}{\mathbf b_{\h,t}}
\DeclareMathOperator{\EDGE}{edge}
\DeclareMathOperator{\FACE}{face}
\DeclareMathOperator{\NODE}{node}
\newcommand{\Vnode}{V^{\NODE}}

\newcommand{\VnodehE}{\Vnode_\h(\E)}
\newcommand{\VnodehF}{\Vnode_\h(\F)}
\newcommand{\Vnodeh}{\Vnode_\h}
\newcommand{\EboldI}{\Ebold_{\mathbf I}}
\newcommand{\cbold}{\mathbf c}
\newcommand{\varepsilonhat}{\widehat \varepsilon}
\newcommand{\muhat}{\widehat\mu}
\newcommand{\sigmahat}{\widehat\sigma}
\newcommand{\vh}{v_\h}
\newcommand{\Ccal}{\mathcal C}
\newcommand{\EE}{\mathcal E^\E}
\newcommand{\EF}{\mathcal E^\F}
\newcommand{\FE}{\mathcal F^\E}
\newcommand{\Fcaln}{\Fcalh}
\newcommand{\Fcalh}{\mathcal F_\h}
\newcommand{\Ecaln}{\Ecalh}
\newcommand{\Ecalh}{\mathcal E_\h}
\newcommand{\psiboldh}{\boldsymbol \psi_\h}
\newcommand{\psiboldI}{\boldsymbol{\psi}_I}
\newcommand{\phiboldh}{\boldsymbol \phi_\h}
\newcommand{\tbold}{\mathbf t}
\newcommand{\tbolde}{\tbold_\e}
\newcommand{\nbold}{\mathbf n}
\newcommand{\nboldE}{\nbold_\E}

\newcommand{\nboldF}{\nbold_\F}
\newcommand{\pbold}{\mathbf p}
\newcommand{\SE}{S^\E}
\newcommand{\SEedge}{\SE_{\EDGE}}
\newcommand{\SEface}{\SE_{\FACE}}
\newcommand{\Ibold}{\mathbf I}
\newcommand{\nboldOmega}{\nbold_\Omega}
\newcommand{\Acal}{\mathcal A}
\newcommand{\Dcal}{\mathcal D}
\newcommand{\Hcal}{\mathcal H}
\newcommand{\XboldN}{\mathbf {X_N}}
\newcommand{\XboldT}{\mathbf {X_T}}

\newcommand{\ts}[1]{t^{#1}} 
\newcommand{\tm}{\ts{m}}
\newcommand{\tmpo}{\ts{m+1}}

\newcommand{\omegam}{\omega^m}

\newcommand{\trcurl}{\text{tr}_{\textbf{curl}}}
\newcommand{\trdiv}{\text{tr}_{\text{div}}}
\newcommand{\D}{D}
\newcommand{\EboldzI}{\Ebold_{0,I}}
\newcommand{\BboldzI}{\Bbold_{0,I}}
\newcommand{\BboldI}{\Bbold_{\mathbf I}}
\newcommand{\vI}{v_I}

\newcommand{\zbold}{\mathbf z}
\newcommand{\Zboldh}{\mathbf Z_h}
\newcommand{\ph}{p_h}

\newcommand{\psibold}{{\bm\psi}}
\newcommand{\phibold}{{\bm\phi}}
\newcommand{\restrict}[2]{{#1}_{|#2}}

\newcommand{\EOD}{\end{document}}
\newcommand{\EOPD}{\end{proof}\end{document}}

\definecolor{MyDarkGreen}{rgb}{0,0.45,0}

\usepackage{changes}
\definechangesauthor[color=red,         name={Franco}]{FD}
\definechangesauthor[color=blue,        name={LorenzoM}]{LM}
\definechangesauthor[color=MyDarkGreen, name={Marco}]{gm}
\definechangesauthor[color=brown,       name={LourenzoB}]{LBdV}

\begin{document}

\begin{frontmatter}
  
\title{Virtual elements for Maxwell's equations}

  \author[BICA] {L.~Beir\~ao~da~Veiga}
  \author[BICB] {, F.~Dassi}
  \author[IMATI]{, G.~Manzini}
  \author[WIEN] {, and L.~Mascotto}

  \address[BICA]{Dip. di Matematica e Applicazioni, Universit\`a degli
    Studi di Milano-Bicocca, Italy (lourenco.beirao@unimib.it)}
  
  \address[BICB]{Dip. di Matematica e Applicazioni, Universit\`a degli
    Studi di Milano-Bicocca, Italy (franco.dassi@unimib.it)}
  
  \address[IMATI]{Istituto di Matematica Applicata e Tecnologie Informatiche, CNR, Pavia, Italy (marco.manzini@imati.cnr.it)}
  
  \address[WIEN]{Fakult\"at f\"ur Mathematik, Universit\"at Wien,
    Austria (lorenzo.mascotto@univie.ac.at)}
  
  \begin{abstract}
We present a low order virtual element discretization
for time dependent Maxwell's equations, which allow for the use of general polyhedral meshes. 
Both the semi- and fully-discrete schemes are considered.
We derive optimal a priori estimates and validate them on a set of numerical experiments.
As pivot results, we discuss some novel inequalities for de~Rahm sequences of nodal, edge, and face virtual element spaces.
\end{abstract}
  

\end{frontmatter}


\raggedbottom
\setcounter{secnumdepth}{4}
\setcounter{tocdepth}{4}


\section{Introduction}
\label{section:introduction}

The Virtual Element Method (VEM) was introduced in \cite{VEMvolley} as a generalization of the Finite Element Method (FEM)
that allows for the use of general polygonal and polyhedral meshes.
Since its introduction, the VEM has shared a wide success in the numerical analysis and engineering communities.
After the introduction of~$H^1$ conforming spaces in~\cite{VEMvolley,equivalentprojectorsforVEM,VEM3Dbasic},
also $H({\rm div})$ and $H({\rm curl})$ conforming spaces in both two and three space dimensions were proposed.
Mixed finite elements for the diffusion problem in mixed form in 2D were introduced in \cite{Brezzi-Falk-Marini,BBMR_generalsecondorder},
while in~\cite{HdivHcurlVEM, da2017virtual, da2018lowest, da2018family}
various families of discrete exact VEM complexes  of $H^1-H({\rm div})-H({\rm curl})-L^2$ type were introduced in 2D and 3D. 
In the above contributions, all such families of spaces are applied to the Kikuchi formulation of the magnetostatic equations, 
used as a simple model problem to showcase the proposed discrete construction.
A recent application for permanent magnet simulations can be found in~\cite{Dassi-DiBarba-Russo}.

On the other hand, finite elements have been widely used for numerical modelling of Maxwell's equations,
a very short representative list being~\cite{Jin:2014, Bermudez-book, 
monk2003finite, Monk1991, Nedelec1980, Monk-Makridakis, zhao2004analysis, Ciarlet-Zou, Assous-Maxwell}.
Important applications involve, for instance, the analysis and design of microwave devices~\cite{Coccioli-Itoh-Pelosi-Silvester:1996},
cavity resonators~\cite{Lee-Wilkins-Mitra:1993,Tierens-DeZutter:2011,Rui-Hu-Liu:2010},
coaxial cables and waveguides~\cite{Wilkins-Lee-Mittra:1991},
antennas and high-power amplifiers~\cite{Teixeira-Bergmann:1997-proc,Teixeira-Bergmann:1997-journal,Greenwood-Jin:1999,Jin-Riley:2009},
electromagnetic scattering~\cite{Khebir-DAngelo-Joseph:1993,MedgyesiMitschang-Putnam:1984}.

Due to the complex geometries that are often faced in many applicative areas of electromagnetism, 
the additional flexibility of general polytopal grids is an important asset, not only in generating an efficient mesh to partition
the domain of interest, but also in handling/gluing/adapting existing meshes. Among the other polytopal technologies, in the realm of electromagnetism it is possible to find (in a nonexhaustive list) polygonal finite elements~\cite{Euler-Schuhmann-Weiland:2006}, mimetic finite differences \cite{MFD-3D-Maxwell}, hybrid high-order methods~\cite{chave2020three}, and discrete exact sequences~\cite{DeRahm-skeletal}.

The aim of the present paper is to use the discrete spaces introduced in \cite{da2018lowest} to develop a virtual element
discretization of the full time-dependent Maxwell's equations.
In order to ease the reader's understanding, we restrict the presentation and analysis to the lowest order case; the generalization of the scheme and the analysis to the general order case, see, e.g., \cite{da2018family}, would follow the same steps.
\paragraph*{Structure of the paper.}
After introducing several Sobolev spaces at the end of this introduction,
we present the model problem in Section~\ref{section:continuous}.
The virtual element schemes for the semi- and fully discrete Maxwell's equations are detailed in Section~\ref{section:VEM};
here, we also address the approximation properties in virtual element and polynomial spaces,
as well as the design of suitable stabilization terms.
We develop convergence estimates for the semi-discrete and the fully discrete cases in the spirit of~\cite{zhao2004analysis},
the latter restricted to the backward Euler case, in Sections~\ref{section:semi-discrete} and~\ref{section:fully-discrete}.
The error estimates show the optimal behaviour of the proposed method.
In order to investigate the practical performance of the scheme, we develop a set of academic numerical tests in Section~\ref{section:numerical:results}.
Eventually, we state some conclusions in Section~\ref{section:conclusions}.
\paragraph*{Notation and functional spaces.}

We employ the standard definitions and notation for Hilbert and
Sobolev spaces~\cite{adamsfournier}.
Given~$s\in\Rbb$ and a Lipschitz domain~$\D$, we denote the Hilbert
space of order~$s$ by~$H^s(\D)$.
We endow~$H^s(\D)$ with the standard inner product, norm, and
seminorms, which we indicate as $(\cdot,\cdot)_{s,\D}$,
$\vert\cdot\vert_{s,\D}$, and $\Vert\cdot\Vert_{s,\D}$.
The special case~$s=0$ consists of the Lebesgue space~$L^2(\D)$ of
real-valued, square integrable functions defined on $\D$.
We define Sobolev spaces of noninteger order by interpolation and Sobolev spaces of negative order by duality.
We analogously consider Sobolev spaces~$H^s(\partial \D)$ on the boundary
$\partial\D$ of~$\D$.

We recall the definition of some differential operators that we shall use in the paper.
Let $\partial_x$, $\partial_y$, and $\partial_z$ denote the partial
derivative along $x$, $y$, and $z$.
Given a two-dimensional vector-valued field~$\vbold=(v_1,v_2):\F\subseteq\Rbb^2\rightarrow\Rbb^2$
and a scalar field~$v:\F\subseteq\Rbb^2\rightarrow\Rbb$, we consider
\begin{align*}
  \divF\vbold := \partial_x v_1 + \partial_y v_2,
  \quad\quad
  \rotF\vbold := \partial_y v_1 - \partial_x  v_2,
  \quad\quad
  \curlboldF v := \big(\partial_y v, -\partial _x v\big)^T.
\end{align*}
In turn, given a three-dimensional vector-valued field~$\vbold=(v_1,v_2,v_3):\E\subseteq\Rbb^3\rightarrow\Rbb^3$, we consider
\begin{align*}
  \div\vbold := \partial_x v_1 + \partial_y v_2 + \partial_z v_3,
  \quad\quad
  \curlbold\vbold := \big(
  \partial_y v_3 - \partial_z v_2, \partial _z v_1 - \partial_x v_3, \partial _x v_2 - \partial_y v_1
  \big)^T.
\end{align*}
For Lipschitz domains~$\D\subset\mathbb R^3$, we introduce the Sobolev~$\curlbold$ and~$\div$ spaces of order~$s>0$
\begin{align*}
  H^s(\curlbold,\D) &:= \left\{ \vbold \in [H^s(\D)]^3 \mid \curlbold \vbold \in [H^s(\D)]^3  \right\},\\[0.25em]
  H^s(\div,\D)      &:=  \left\{ \vbold \in [H^s(\D)]^3 \mid \div \vbold \in H^s(\D)          \right\}.
\end{align*}
If~$s=0$, we write~$H(\curlbold,\D)$ and~$H(\div,\D)$.
We denote the unit vector that is orthogonal to the
boundary~$\partial\D$ and pointing out of~$\D$ by~$\nbold_\D$.
Furthermore, we recall the existence of the two trace
operators~$\trcurl: H(\curlbold,\D)\rightarrow [H^{-\frac12}(\partial\D)]^2$ and~$\trdiv : H(\div,\D)\rightarrow
H^{-\frac12}(\partial\D)$ such
that~$\trcurl(\vbold)=\vbold\times\nbold_{\D}$
and~$\trdiv(\vbold)=\psibold\cdot\nbold_{\D}$ for
all~$\vbold \in H(\curlbold, \D)$ and~$\psibold\in H(\div,\D)$,
respectively; see, e.g., \cite[Section~3.5]{monk2003finite}.
According to the standard notation, $L^{\infty}(\D)$ is the Sobolev
space of functions that are bounded almost everywhere and
~$W^{1,\infty}(\D)$ the Sobolev space of functions in $L^{\infty}(\D)$
whose first weak derivatives are also in $L^{\infty}(\D)$.
We shall also consider~$\curlbold$- and~$\div$-spaces with zero
boundary conditions such as
\begin{align*}
  H_0(\curlbold,\D) &:= \left\{ \vbold\in H(\curlbold,\D) \mid \trcurl(\vbold)= 0 \text{~on~} \partial \D \right\},\\
  H_0(\div,\D)      &:= \left\{ \vbold\in H(\div,\D)      \mid \trdiv(\vbold) = 0 \text{~on~} \partial \D \right\}.
\end{align*}
Let~$X$ denote a scalar or vector Sobolev space of any order over the
domain~$\D\subset\Rbb^3$; $(a,b)$ an open, connected subset of $\Rbb$,
and $p$ a real number in the interval $[1,\infty]$.
The Bochner space~\cite{evansPDE}~$L^p((a,b),X)$ is the vector space
of functions~$\vbold$ with finite norm
\begin{align*}
  \Vert\vbold\Vert_{L^p((a,b), X)}:=
  \begin{cases}
    \left(\int_a^b \Vert \vbold(t) \Vert_{X}^p \text{dt}\right)^{\frac1p}     & \text{if } p\in [1,\infty)\\
      \text{essSup}_{(a,b)} \Vert \vbold(\cdot) \Vert_{X}    & \text{otherwise}.
  \end{cases}
\end{align*}
Finally, for any two positive quantities~$a$ and~$b$, we
write~$a\lesssim b$ and~$a\gtrsim b$ if there exists a positive
constant~$c$ such that $a\le c\,b$ and~$a\ge c\,b$, respectively.
We also write~$a\approx b$ if $a\lesssim b$ and~$b\lesssim a$.
We require the constant~$c$ to be independent of the discretization
parameters.
In the following proofs, the explanation of the identities and upper and lower bounds will appear either in the preceeding text
or as an equation reference above the equality symbol ``$=$'' or the inequality symbols ``$\leq$'', ``$\geq$'' etc,
whichever we believe it is easier for the reader.

\section{The continuous problem}
\label{section:continuous}

We consider the strong form of Maxwell's equations on a polyhedral domain~$\Omega\subset\Rbb^3$ with Lipschitz boundary~$\partial\Omega$:
\emph{Given the initial data~$\Ebold^0$ and~$\Bbold^0$, find the
electric field~$\Ebold$ and the magnetic induction field~$\Bbold$ such
that}
\begin{equation}
  \label{eq:Maxwell:strong} 
  \begin{cases}
    \varepsilon\Eboldt + \sigma\Ebold - \curlbold(\mu^{-1}\Bbold)
    = \Jbold
    & \quad\text{in}~\Omega, \,\forall t\in[0,T],\\[0.5em]
    \Bboldt + \curlbold \Ebold = \mathbf{0}
    & \quad\text{in}~\Omega, \,\forall t\in[0,T],\\[0.5em]
    \Ebold(0) = \Ebold^0,\quad
    \Bbold(0) = \Bbold^0 \quad
    & \quad\text{in}~\Omega,\\[0.5em]
    \Ebold\times\nboldOmega = 0,\quad
    \Bbold\cdot \nboldOmega = 0
    & \quad\text{on}~\partial\Omega,
  \end{cases}
\end{equation}
where the subscript $t$ denotes the first derivative in
time (so throughout the paper we use $q_t$ instead of $\partial
q\slash{\partial t}$ for a given time-dependent quantity $q(t)$).
Above, $\Jbold$, $\varepsilon$, $\sigma$, and~$\mu$ denote the electric current density that is externally applied to the system, the electric permittivity, the electric conductivity, and the magnetic permeability.
We consider homogeneous boundary conditions to ease the exposition, since
the nonhomogeneous boundary case presents  further complications.
We assume that the initial magnetic induction is a solenoidal field, i.e.,
\begin{equation}
  \div\Bbold^0 = 0.
  \label{eq:assumption:0div:initial-induction} 
\end{equation}
Taking the divergence of the second equation in~\eqref{eq:Maxwell:strong}, we readily deduce that
\begin{align}
  \div \Bbold(t) = 0
  \quad\quad\forall t\in[0,T].
  \label{eq:assumption:0div:all-times} 
\end{align}
The weak formulation of Maxwell's equations reads as follows:
\begin{equation}
  \begin{cases}
    \text{Find}~(\Ebold,\Bbold)\in H_0(\curlbold,\Omega)\times H_0(\div,\Omega) \text{~such~that}\\[0.5em]
    (\varepsilon\Eboldt + \sigma\Ebold,\wbold)_{0,\Omega} - (\mu^{-1}\Bbold,\curlbold\wbold)_{0,\Omega}
    = (\Jbold,\wbold)_{0,\Omega}
    &\quad\forall\wbold\in H_0(\curlbold,\Omega),\\[0.5em]
    (\mu^{-1}\Bboldt,\psibold)_{0,\Omega} + (\mu^{-1}\psibold,\curlbold\Ebold)_{0,\Omega}
    = 0
    &\quad\forall\psibold\in H_0(\div,\Omega).
  \end{cases}
  \label{eq:Maxwell:weak} 
\end{equation}
In the sequel, we shall assume that there exist strictly positive constants $\sigma^\star, \varepsilon_\star, \varepsilon^\star, \mu_\star, \mu^\star$ such that, for all $x \in \Omega$, the 
material parameters satisfy
\begin{equation}\label{eq:param:bound}
0 \le \sigma(x) \le \sigma^\star \ , \quad
\varepsilon_\star \le \varepsilon \le \varepsilon^\star \ , \quad
\mu_\star \le \mu(x) \le \mu^\star \ .
\end{equation}
\paragraph*{On the regularity of the solutions to~\eqref{eq:Maxwell:weak}.}
Under suitable assumptions on the regularity of the data, problem~\eqref{eq:Maxwell:weak} admits a unique solution; see, e.g., \cite[Theorem~2.1]{zhao2004analysis} and the references therein.
We here recall sufficient conditions from~\cite{zhao2004analysis} leading to extra smoothness in space for the solutions to Maxwell's equations that will be needed in the following derivations.
To the aim, given~$(\ubold,\vbold) \in \Hcal:=[L^2(\Omega)]^3 \times [L^2(\Omega)]^3$,
we first introduce an operator~$\Acal$ with domain
\[
\Dcal(\Acal) = 
\left\{  
\begin{pmatrix} \ubold \\ \vbold \end{pmatrix} \in \Hcal \text{ such that }
\ubold \in H_0(\curlbold,\Omega) \text{ and } \curlbold(\mu^{-1} \vbold) \in [L^2(\Omega)]^3 \right\},
\]
where the operator~$\Acal$ is given by
\[
\Acal \begin{pmatrix} \ubold \\ \vbold \end{pmatrix}
= \begin{pmatrix} -\varepsilon^{-1} \curlbold(\mu^{-1}\vbold) + \varepsilon^{-1} \sigma \ubold \\ \curlbold \ubold \end{pmatrix}
\qquad
\forall \begin{pmatrix} \ubold \\ \vbold \end{pmatrix} \in \Dcal(\Acal).
\]
Introduce
\[
\XboldN(\Omega,\varepsilon):= \{ \ubold \in H_0(\curlbold,\Omega) \mid \varepsilon \ubold \in H(\div,\Omega) \},\qquad  \XboldT(\Omega,\mu):= \{ \ubold \in H(\curlbold,\Omega) \mid \mu \ubold \in H_0(\div,\Omega) \}.
\]
Let~$\Ebold$ and~$\Bbold$ be the solutions to~\eqref{eq:Maxwell:weak}.
Assume that~$\mu$ or~$\varepsilon$ is a constant function, $\sigma$, $\mu$, and~$\varepsilon$ are continuous, and $\sigma / \varepsilon \in W^{1,p}(\Omega)$ with $p > 3$.
Further assume
\[
\begin{split}
& \Jbold \in \mathcal C^3([0,T], [L^2(\Omega)]^3), \qquad
\div \Jbold \in C^2([0,T], [L^2(\Omega)]^3), \\
& (\Jbold(0),0)^T ,\; \Acal (\Jbold(0),0)^T + (\Jbold_t(0),0)^T \in
\XboldN(\Omega,\varepsilon) \times \mu \XboldT(\Omega,\mu).
\end{split}
\]
Then, as in~\cite[Theorem~$2.3$]{zhao2004analysis},
we have that $\Ebold(t)$, $\Bbold(t)$, $\Eboldt(t)$, $\curlbold \Ebold (t)$, and $\curlbold \Eboldt$ belong to $H^s(\Omega)$, for some~$s>1/2$, for all~$t \in [0,T]$.\

\section{The virtual element method} \label{section:VEM}

In this section, we construct the virtual element method for the variational formulation of Maxwell's equations~\eqref{eq:Maxwell:weak}
and discuss its main properties.
We formulate the VEM on sequences of polyhedral meshes, whose properties are discussed in Section~\ref{section:meshes}.
In Sections~\ref{subsection:nodal-spaces}--\ref{subsection:face-spaces},
we briefly review the definitions of the lowest-order nodal, edge and
face virtual element spaces and the design of computable discrete
bilinear forms.
In Section~\ref{subsection:exact-sequence}, we recall from~\cite{da2018lowest} that these spaces form an exact de~Rham sequence and review some related property.
The design of the virtual element spaces follows the guidelines
of~\cite{da2018lowest}; see also~\cite{HdivHcurlVEM, da2017virtual,
  da2018family}.
In Sections~\ref{subsection:semi-discrete}
and~\ref{subsection:fully-discrete}, we present the semi-discrete and
fully-discrete method.

\subsection{Polyhedral meshes and mesh assumptions} \label{section:meshes}
Let
$\{ \taun \}_{\h}$ be a sequence of mesh partitionings of the computational domain~$\Omega$ labeled by the subscript~$h$,
which stands for the \emph{mesh size parameter}.
Every mesh~$\taun$ is a collection of open, bounded, simply connected polyhedral elements $\E$ such that $\overline{\Omega}=\bigcup_{\E\in\taun}\overline{\E}$.
The mesh elements are nonoverlapping in the sense that the intersection of any possible pair of
them can only be either the empty set, a set of common vertices, or a shared portion of their boundaries (which is a union of edges).
The mesh size parameter is defined as $h=\max_{\E\in\taun}\hE$, where
$\hE=\sup_{\xbold,\xbold'\in\E}|\xbold-\xbold'|$ is the
\emph{diameter} of $\E$.
Other characteristic lengths are the face diameters $\hF=\sup_{\xbold,\xbold'\in\F}|\xbold-\xbold'|$,
which are defined for any mesh face $F$, and the edge lengths~$\he$, which are defined for any mesh edge $\e$.
For all~$\h$, we denote the set of faces and edges by~$\Fcaln$ and~$\Ecaln$.
Moreover, we denote the set of faces of an element $\E\in\taun$ by
$\FE$ and the set of edges of a face~$F\in\Fcaln$ by $\EF$.
Consistently with our previous notation, $\nboldE$ is the unit vector
pointing out of element~$\E$, and $\bboldE$ and~$\bboldF$ are the
centroids of $\E$ and~$\F$.

Let~$\gamma >0$. A face~$\F$ is said to be \emph{$\gamma$-shape regular} if there
exists a a two-dimensional ball~$B$ with diameter $\h_B$ in the interior of~$\F$ such that $\h_B>\gamma\hF$.
Similarly, an element~$\E$ is said to be \emph{$\gamma$-shape regular}
if there exists a three-dimensional ball~$B$ with diameter $\h_B$ in the interior of~$\E$ such that $\h_B>\gamma\hE$.

\medskip
In the rest of the manuscript we assume that all the meshes~$\taun$ of a given sequence
$\{\taun\}$ satisfy these conditions uniformly: there exists a real constant factor~$\gamma \in (0,1)$ independent of $h$ such that
\begin{itemize}
\smallskip
\item all the elements~$\E\in\taun$ and faces~$\F\in\Fcaln$ are $\gamma$-shape regular;
\smallskip
\item $\gamma\hE \le \hF$ for every~$\F\in\FE$ of every element~$\E\in\taun$, and, analogously, $\gamma \hF \le \he$
for every edge~$\e\in\EF$ of every face~$\F\in\Fcaln$.
\end{itemize}
\medskip

We assume that the (scalar and real valued) problem coefficients~$\varepsilon$, $\sigma$,
and~$\mu$ in~\eqref{eq:Maxwell:strong} and~\eqref{eq:Maxwell:weak} are
piecewise continuous over~$\taun$.
As a consequence, we can approximate them by the three piecewise constant
functions~$\varepsilonhat$, $\sigmahat$, and~$\muhat$ given by, in every mesh element $\E\in\taun$,
\begin{equation} 
  \varepsilonhat_{\E} := \varepsilon(\bboldE),\quad
  \sigmahat_{\E}      := \sigma(\bboldE), \quad
  \muhat_{\E}         := \mu(\bboldE).
\end{equation}
To perform the analysis of the method, we also need the additional
regularity condition that, for every element $\E\in\taun$,
\begin{equation}
\restrict{\mu^{-1}}{\E},\,\,\, \restrict{\sigma}{\E},\,\, \restrict{\varepsilon}{\E} \in W^{1,\infty}(\E).
  \label{assumption:coefficients}
\end{equation}

\medskip
On every mesh~$\taun$, we consider the broken Sobolev space of order~$s\ge 0$
\begin{align*}
H^s(\taun) := \Big\{ \vbold\in[L^2(\Omega)]^3 \mid \restrict{\vbold}{\E}\in[H^s(\E)]^3\,\,\forall\E\in\taun \Big\},
\end{align*}
endowed with the seminorm
\begin{align*}
  \vert\vbold\vert_{s,\taun}^2 := \sum_{\E\in\taun}\vert\vbold\vert^2_{s,\E}.
\end{align*}
\medskip

\noindent For all elements~$\E$, we define the local $L^2$-orthogonal
projector~$\Pizboldh:[L^2(\E)]^3\rightarrow[\Pbb_0(\E)]^3$ onto constant vectors as
\begin{equation} \label{L2-projector}
  \begin{split}
    ( \Pizboldh \Ebold, \cbold)_{0,\E} =
    (\Ebold, \cbold)_{0,\E}
    \quad\quad\forall\Ebold\in[L^2(\E)]^3,\,\forall\cbold\in[\Pbb_0(\E)]^3.
  \end{split}
\end{equation}
Given a function~$\Ebold \in [H^s(\E)]^3$, $0<s\le 1$, we have the standard approximation property
\[
\Vert \Ebold - \Pizboldh \Ebold \Vert_{0,\E} \lesssim \hE^s \vert \Ebold \vert_{s,\E}.
\]

\subsection{Nodal virtual element spaces} \label{subsection:nodal-spaces}
Consider a mesh face $\F\in\Fcalh$ and set
\begin{equation}
  \xboldF = \xbold - \bboldF \quad\quad\forall\xbold\in\F .
  \label{shifted-face}
\end{equation}
The nodal virtual element space on face~$\F$ is
\begin{align*}
  \VnodehF :=
  \bigg\{
  \vh\in\Ccal^0(\overline\F)
  \mid
  \Delta\vh\in\Pbb_{0}(\F),\;
  \restrict{\vh}{\e}\in\Pbb_1(\e)\,\,\forall\e\in\EF,\;
  \int_\F\nabla\vh\cdot\xboldF = 0
  \bigg\}.
\end{align*}
We use $\VnodehF$ in the definition of the nodal virtual element space
on element~$\E\in\taun$, which is given by
\begin{align*}
  \VnodehE :=
  \Big\{
  \vh\in\Ccal^0(\overline\E)
  \mid
  \Delta\vh = 0,\;
  \restrict{\vh}{\F}\in\VnodehF\,\,\forall\F\in\EF
  \Big\}.
\end{align*}
Every virtual element function $\vh\in\VnodehE$ is uniquely
characterized by the set of its values at the vertices of $\E$, which
we take as the degrees of freedom . 
This unisolvence property is proved in~\cite{da2018lowest}.
Then, we introduce the global virtual element space of the functions
that are globally defined on the computational domain $\Omega$ and
have zero trace on $\partial\Omega$:
\begin{align*}
  \Vnodeh :=
  \Big\{
  \vh\in H^1_0(\Omega)
  \mid
  \restrict{\vh}{\E}\in\VnodehE\,\,\forall\E\in\taun
  \Big\}.
\end{align*}
The degrees of freedom of~$\Vnodeh$ are given by an $H^1$-conforming coupling of the local degrees of freedom, i.e., collecting all the internal vertex values.

For future reference, we also define the interpolant~$\vI\in\Vnodeh$ of a function $v\in\Ccal^0(\overline{\Omega})$
as the unique nodal virtual element function with the same vertex values as~$v$.
Formally, this definition reads as
\begin{align}
  \vI(\nu) := v(\nu) \quad \quad \text{for all vertices~$\nu$ in~$\taun$.}
  \label{interpolant:nodal}
\end{align}
Upper bounds for the approximation error $v-\vI$ are available in the literature; see, e.g., \cite{VEMchileans,beiraolovadinarusso_stabilityVEM,brennerVEMsmall}.

\subsection{Edge virtual element spaces}
\label{subsection:edge-spaces}
\paragraph*{Space definitions.}
The edge virtual element space on face $\F\in\Fcalh$ is
\begin{equation}
  \label{local:edge:element}
  \begin{split}
    \VboldedgehF :=
    \bigg\{
    \vboldh\in[L^2(\F)]^2 \mid 
    & \divF\vboldh\in\Pbb_0(\F),\;\rotF\vboldh\in\Pbb_0(\F),\\
    & \,\vboldh\cdot\tbolde\in\Pbb_0(\e)\,\,\forall\e\in\EF,\;\int_\F\vboldh\cdot\xboldF = 0
    \bigg\},
  \end{split}
\end{equation}
where~$\xboldF$ is defined as in~\eqref{shifted-face}.
Next, we define the edge virtual element space on an element~$\E\in\taun$ as
\begin{equation} \label{local:edge-element:polyhedron}
\begin{split}
\VboldedgehE :=
\bigg\{ \vboldh
& \in [L^2(\E)]^3 \mid     \div\vboldh=0,\;    \curlbold\curlbold\vboldh\in[\Pbb_0(\E)]^3,\\    
&(\nboldF\times\vboldh{}_{|\F})\times\nboldF \in\VboldedgehF\,\,\forall\F\in\FE,\quad 
\vboldh\cdot\tbolde\text{~continuous at each edge~$\e$,~}\\[1.em]
& \int_\E\curlbold\vboldh\cdot(\xboldE\times\pbold_0)=0\,\,\forall\pbold_0\in[\Pbb_0(\E)]^3 \bigg\},
\end{split}
\end{equation}
where~$\xboldE = \xbold - \bboldE$ for all~$\xbold\in\E$.
We note that $(\nboldF\times\vboldh{}_{|\F})\times\nboldF$
corresponds to the projection of $\vboldh{}_{|\F}$ onto the tangent plane to $\F$.
The last integral constraints in~\eqref{local:edge:element} and~\eqref{local:edge-element:polyhedron} are required to allow for the computation of the $L^2$ projector onto vector constant functions defined in~\eqref{L2-projector}; see \cite[Section 4.1.2]{da2018lowest} for more details.

Every virtual element function $\vboldh\in\VboldedgehE$ is uniquely characterized by the constant values of~$\vboldh\cdot\tbolde$ on the
edges~$\e$ of~$\E$, which we take as the degrees of freedom.
The unisolvence of this set of degrees of freedom for the space~$\VboldedgehE$ is proved, e.g., in~\cite{HdivHcurlVEM}.
A noteworthy property of the local edge virtual element space is that the $L^2$-orthogonal projector~$\Pizboldh$ defined
in~\eqref{L2-projector} is computable from the degrees of freedom of the edge virtual element functions; see, e.g.,~\cite{HdivHcurlVEM}.
We define the global edge virtual element space as
\begin{align*}
  \Vboldedgeh :=
  \Big\{   \vboldh \in H_0(\curlbold,\Omega)   \mid   \vboldh{}_{|\E}\in\VboldedgehE\,\,\forall\E\in\taun  \Big\}.
\end{align*}
This definition includes the homogeneous boundary conditions on
$\partial\Omega$.
The global set of degrees of freedom of~$\Vboldedgeh$ is obtained via
an $H(\curlbold)$-conforming coupling of the local ones.

\paragraph*{Bilinear forms.}
As customary in the virtual element framework, we introduce local computable bilinear forms
\[[\cdot, \cdot]_{\EDGE,\E}: \VboldedgehE \times \VboldedgehE \rightarrow \Rbb,\]
mimicking the $L^2$ inner product~$(\cdot,\cdot)_{0,\E}$.
In particular, we first introduce the stabilizing bilinear form
$\SEedge:\VboldedgehE\times\VboldedgehE\rightarrow\Rbb$ satisfying
\begin{equation} 
  \SEedge(\vboldh,\vboldh)
  \approx
  \Vert\vboldh\Vert^2_{0,\E}
  \quad\quad \forall\vboldh\in\ker(\Pizboldh)\cap\VboldedgehE.
  \label{preliminary:stab}
\end{equation}
Then, we define the local discrete counterpart of the $L^2$ inner product as
\begin{equation}
  \big[\uboldh,\vboldh\big]_{\EDGE,\E} :=
  (\Pizboldh \uboldh, \Pizboldh \vboldh)_{0,\E} +
  \SEedge( (\Ibold-\Pizboldh)\uboldh, (\Ibold-\Pizboldh)\vboldh )
  \quad \forall\uboldh,\,\vboldh\in\VboldedgehE.
  \label{discrete:bf:edge}
\end{equation}
The local discrete bilinear forms~$[\cdot, \cdot]_{\EDGE,\E}$
satisfies the stability condition
\begin{equation} 
  \big[\vboldh,\vboldh\big]_{\EDGE,\E}
  \approx
  \Vert\vboldh\Vert^2_{0,\E}
  \quad\quad \forall\vboldh\in\VboldedgehE
  \label{stability:edge}
\end{equation}
and the consistency property
\begin{equation} 
  \big[\pbold_0,\vboldh\big]_{\EDGE,\E} =
  (\pbold_0,\vboldh)_{0,\E}
  \quad\quad
  \forall\pbold_0\in[\Pbb_0(\E)]^3,\,\vboldh\in\VboldedgehE.
  \label{consistency:edge}
\end{equation}
Whereas property~\eqref{consistency:edge} follows from
definition~\eqref{discrete:bf:edge}, property~\eqref{stability:edge}
requires the design of a suitable stabilization satisfying~\eqref{preliminary:stab}.
If we consider the stabilization
\begin{equation}
  \SEedge(\uboldh,\vboldh) :=
  \hE^2\sum_{\F\in\FE}\sum_{\e\in\EF} (\uboldh \cdot \tbolde, \vboldh \cdot \tbolde)_{0,\e}
  \quad\quad \forall\uboldh,\,\vboldh\in\VboldedgeE,
  \label{explicit:stab:edge}
\end{equation}
proposed in~\cite[formula~(4.8)]{da2018lowest}, then the stability bounds~\eqref{stability:edge} are proven in~\cite[Proposition~5.5]{face-edge-VEM-interpolation-low}.

We introduce the global discrete bilinear forms
\begin{align}
  [\varepsilonhat \uboldh, \vboldh]_{\EDGE}
  &:=
  \sum_{\E\in\taun}\varepsilonhat_{|\E} [\uboldh,\vboldh]_{\EDGE,\E}
  \quad\quad \forall\uboldh,\,\vboldh\in\Vboldedgeh,
  \label{eq:norm:edge:varepsilon}
  \intertext{and}
  [\sigmahat \uboldh, \vboldh]_{\EDGE}
  &:=
  \sum_{\E\in\taun} \sigmahat_{|\E} [\uboldh,\vboldh]_{\EDGE,\E}
  \quad\quad \forall\uboldh,\,\vboldh\in\Vboldedgeh.
  \label{eq:norm:edge:inv-mu}
\end{align}
In these definitions, we scale the local bilinear forms in the
right-hand side of~\eqref{discrete:bf:edge} by~$\varepsilonhat$
and~$\sigmahat$.
Moreover, in the forthcoming analysis, we shall employ the mesh-dependent norm
\begin{align*}
  \Vertiii{\vboldh}_{\EDGE}^2:= [\vboldh,\vboldh]_{\EDGE},
\end{align*}
which is induced by the scalar product defined
in~\eqref{eq:norm:edge:varepsilon} by setting~$\varepsilonhat=1$.

\paragraph*{Interpolation properties.}
We denote the interpolation in~$\Vboldedgeh$ of a given vector-valued
field $\vbold \in H^s(\curlbold, \Omega)$, $1/2 < s \leq 1$ by~$\vboldI$.
By definition, $\vboldI$ is the only function in~$\Vboldedgeh$ such
that
\begin{equation} \label{interpolant:edge}
  \int_\e (\vbold - \vboldI) \cdot \tbolde = 0 \quad \quad \forall \e \in \Ecaln.
\end{equation}
We recall the following interpolation result; see~\cite[Proposition~4.5;~Corollary~4.6]{face-edge-VEM-interpolation-low}.

\medskip
\begin{prop} \label{prop:interpolation:edge}
Let~$\vbold \in H^s(\curlbold, \Omega)$, $1/2<s\le1$, and $\vboldI\in \Vboldedgeh$ be its interpolant as defined in~\eqref{interpolant:edge}.
Then, for all~$\E \in \taun$, we have that that
  \begin{equation} \label{interpolation-estimates:edge}
    \begin{split}
      \Vert \vbold - \vboldI \Vert_{0,\E}
      &\lesssim
      \hE^s \vert \vbold \vert_{s,\E} + \hE \Vert \curlbold \vbold \Vert_{0,\E} + \hE^{s+1} \vert \curlbold \vbold \vert_{s,\E},\\[0.5em]
      \Vert \curlbold(\vbold - \vboldI) \Vert_{0,\E}
      &\lesssim
      \hE^s \vert\curlbold \vbold \vert_{s,\E}.
    \end{split}
  \end{equation}
\end{prop}

We shall use inequalities~\eqref{interpolation-estimates:edge} in the
forthcoming analysis of Sections~\ref{section:semi-discrete} and~\ref{section:fully-discrete}.

\subsection{Face virtual element spaces}  \label{subsection:face-spaces}
\paragraph*{Space definitions.}
Given an element~$\E \in \taun$, we define the face virtual element
space as
\begin{align*}
\begin{split}
\VboldfacehE := \bigg\{ \psiboldh
& \in[L^2(\E)]^3 \mid \div\psiboldh\in\Pbb_0(\E),\;\curlbold\psiboldh\in[\Pbb_0(\E)]^3,\\
& \psiboldh\cdot\nboldF\in\Pbb_0(\F)\,\forall\F\in\EE,\\
& \int_\E \psiboldh \cdot (\xboldE \times \pbold_0) = 0 \ \forall\pbold_0 \in [\Pbb_0(\E)]^3 \bigg\},
\end{split}
\end{align*}
where we recall that $\xboldE$ is defined as $\xbold-\bboldE$ for
all~$\xbold\in\E$.
Every virtual element function $\psiboldh\in\VboldfacehE$ is uniquely
characterized by the constant values of~$\psiboldh\cdot\nboldF$ on the
faces~$\F$ of~$\FE$, which we take as the degrees of freedom.
The unisolvence of this set of degrees of freedom for the space~$\VboldfacehE$ is proved in~\cite{HdivHcurlVEM}.
A noteworthy property of the local face virtual element space is that
the $L^2$-orthogonal projector~$\Pizboldh$ defined
in~\eqref{L2-projector} is computable from the degrees of freedom
of the face virtual element functions; see, e.g.,~\cite{HdivHcurlVEM}.
We define the global face virtual element space as
\begin{align*}
  \Vboldfaceh :=
  \bigg\{
  \psiboldh \in H_0(\div,\Omega)
  \mid
  \psiboldh{}_{|\E}\in\VboldfacehE \,\,\forall\E\in\taun
  \bigg\}.
\end{align*}
This definition includes homogeneous boundary conditions on~$\partial\Omega$.
The set of degrees of freedom of~$\Vboldfaceh$ is obtained via an
$H(\div)$-conforming coupling of the local ones, i.e., by collecting
together the internal degrees of freedom.

\paragraph*{Bilinear forms.}
As in the edge element case, we introduce local computable bilinear forms
\[\big[\cdot,\cdot\big]_{\FACE,\E}:\VboldfacehE\times\VboldfacehE\rightarrow\Rbb,\]
mimicking the $L^2$ inner product $(\cdot,\cdot)_{0,\E}$ on~$\VboldfacehE$.
In particular, we first introduce the stabilizing bilinear
form~$\SEface:\VboldfacehE\times\VboldfacehE\rightarrow\Rbb$
satisfying
\begin{equation} \label{stability:preliminary:face}
  \SEface(\psiboldh,\psiboldh)
  \approx \Vert\psiboldh\Vert^2_{0,\E}
  \quad\quad \forall\psiboldh\in\ker(\Pizboldh)\cap\VboldfacehE.
\end{equation}
Then, we define the local discrete counterpart of the $L^2$ inner
product as
\begin{equation} \label{discrete:bf:face}
  \big[\psiboldh,\phiboldh\big]_{\FACE,\E} :=
  (\Pizboldh\psiboldh,\Pizboldh\phiboldh)_{0,\E} +
  \SEface( (\Ibold-\Pizboldh)\psiboldh, (\Ibold-\Pizboldh)\phiboldh )
  \quad \forall\psiboldh,\,\phiboldh\in\VboldfacehE.
\end{equation}
The local discrete bilinear forms~$[\cdot,\cdot]_{\FACE,\E}$ satisfies the stability condition
\begin{equation}
  \label{stability:face}
  \big[\psiboldh,\psiboldh\big]_{\FACE,\E}
  \approx
  \Vert\psiboldh\Vert^2_{0,\E}
  \quad\quad \forall\psiboldh\in\VboldfacehE
\end{equation}
and the consistency property
\begin{equation}
  \label{consistency:face}
  \big[\pbold_0,\psiboldh\big]_{\FACE,\E} =
  (\pbold_0,\psiboldh)_{0,\E}
  \quad\quad \forall\pbold_0\in[\Pbb_0(\E)]^3,\,\psiboldh\in\VboldfacehE. 
\end{equation}
Whereas~\eqref{consistency:face} follows from
definition~\eqref{discrete:bf:face}, property~\eqref{stability:face}
requires the design of a suitable stabilization satisfying~\eqref{stability:preliminary:face}.
We consider the stabilization, cf.~\cite[(4.17)]{da2018lowest},
\begin{equation}
  \label{explicit:stab:face}
  \SEface(\psiboldh,\phiboldh) :=
  \hE\sum_{\F\in\FE} (\nboldF \cdot  \psiboldh, \nboldF \cdot \phiboldh )_{0,\F}
  \quad\quad
  \forall\psiboldh,\,\phiboldh\in\VboldfaceE.
\end{equation}
The stability bounds~\eqref{stability:face} are proven
in~\cite[Proposition~5.2]{face-edge-VEM-interpolation-low}.
Finally, we introduce the global discrete bilinear form
\begin{equation}
  \big[\muhat^{-1}\psiboldh,\phiboldh\big]_{\FACE} :=
  \sum_{\E\in\taun} \muhat^{-1}_{|\E} [\psiboldh,\phiboldh]_{\FACE,\E}
  \quad\quad \forall\psiboldh,\,\phiboldh\in\Vboldfaceh.
  \label{scaled:bf:face}
\end{equation}
In the forthcoming analysis, we shall employ the mesh-dependent norm
\begin{align*}
  \Vertiii{\psiboldh}_{\FACE}^2 := \big[\psiboldh,\psiboldh\big]_{\FACE}
  \quad \forall\psiboldh\in\Vboldfaceh,
\end{align*}
which is induced by the scalar product defined
in~\eqref{scaled:bf:face} by setting~$\muhat=1$.

\paragraph*{Interpolation properties.}
We denote the interpolation in~$\Vboldface$ of a given vector-valued
field ~$\psibold \in [H^s(\Omega)]^3$, $1/2 < s \leq 1$
by~$\psiboldI$.
By definition, $\psiboldI$ is the only function in~$\Vboldface$ such
that
\begin{equation}
  \label{interpolant:face}
  \int_\F (\psibold - \psiboldI) \cdot \nboldF = 0 \quad \quad \forall \F \in \Fcaln.
\end{equation}
We recall the following interpolation result; see~\cite[Proposition
  3.2; Corollary 3.3]{face-edge-VEM-interpolation-low}.

\medskip
\begin{prop}
  \label{prop:interpolation:face}
  Let~$\psibold\in [H^s(\div,\Omega)]^3$, $1/2<s\leq1$, and~$\psiboldI$ be
  defined as in~\eqref{interpolant:face}.
  Then, for all $\E\in\taun$, we find that
  \begin{equation}
    \label{interpolation-estimates:face}
    \begin{split}
      & \Vert \psibold - \psiboldI \Vert_{0,\E}
      \lesssim
      \hE^s \vert\psibold\vert_{s,\E},
      \quad\quad \Vert \div(\psibold - \psiboldI) \Vert_{0,\E}
      \lesssim
      \hE^s \vert \div \psi \vert_{s,\E}.
    \end{split}
  \end{equation}
\end{prop}

We shall use inequalities~\eqref{interpolation-estimates:face} in Sections~\ref{section:semi-discrete} and~\ref{section:fully-discrete}.

\subsection{Exact sequence properties}
\label{subsection:exact-sequence}
\noindent
We set
\begin{equation} \label{vtilde-space-edge}
  \Vtildeboldedgeh :=
  \left\{
  \vboldh \in \Vboldedgeh \mid \curlbold(\vboldh)=0 \right\}.
\end{equation}
As observed in~\cite[equation~(4.33)]{da2018lowest}, the following
identity is valid:
\begin{equation} \label{exact:sequence:VEM-edge}
  \Vtildeboldedgeh = \nabla(\Vnodeh).
\end{equation}
Analogously, we set
\begin{equation} \label{Vtilde-space}
  \Vtildeboldfaceh :=
  \left\{
  \psiboldh\in\Vboldfaceh
  \mid
  \div \psiboldh =0
  \right\}.
\end{equation}
As observed in~\cite[equation~(4.35)]{da2018lowest}, the following
identity is also valid:
\begin{equation} \label{exact:sequence:VEM}
  \Vtildeboldfaceh = \curlbold(\Vboldedgeh).
\end{equation}
In particular, the spaces~$\Vnodeh$, $\Vboldedgeh$, and~$\Vboldfaceh$
form an exact sequence; see, e.g., \cite{HdivHcurlVEM,da2018lowest}.

\medskip
\begin{remark} \label{remark:interpolation-exact}
  As shown in~\cite{da2018lowest}, the following commuting diagram properties are valid:
  \smallskip
  \begin{itemize}
  \item $\div\BboldI =\Pizh(\div \Bbold)$ for a given vector-valued
    field~$\Bbold\in H^s(\div,\Omega)$, $1/2 < s \le 1$, where~$\Pizh$
    is the scalar version of the $L^2$ projector
    in~\eqref{L2-projector} and~$\BboldI$ is the face interpolation
    of~$\Bbold$ defined in~\eqref{interpolant:face};

    \smallskip
  \item $\curlbold \EboldI = (\curlbold \Ebold)_I$ for a given
    vector-valued field~$\Ebold\in H^s(\curlbold,\Omega)$, $1/2<s \le
    1$, where on the left- and right-hand sides we consider the edge
    and face interpolations of~$\Ebold$ and~$\curlbold(\Ebold)$ that
    are respectively defined in~\eqref{interpolant:edge}
    and~\eqref{interpolant:face};

    \smallskip
  \item $\nabla \vI = (\nabla v)_I$ for a given scalar field~$v\in
    H^1(\Omega)$, where on the left- and right-hand sides we consider
    the nodal and edge interpolations of~$v$ and~$\nabla v$ that are
    respectively defined in~\eqref{interpolant:nodal}
    and~\eqref{interpolant:edge}.
  \end{itemize}
\end{remark}

\subsection{Two novel operators}
We introduce two novel operators on the spaces~$\Vboldedgeh$
and~$\Vboldfaceh$, which also satisfy a commuting diagram property;
see Proposition~\ref{prop:commuting-diagram} below.
We begin by defining the weighted, global projector~$\Piboldh :
H(\curlbold,\Omega) \rightarrow \Vboldedgeh$ as
\begin{equation}
  \label{curl-projector-edge}
  \begin{split}
    & \begin{cases}
        [\muhat^{-1} \curlbold(\Piboldh \Ebold), \curlbold \wboldh]_{\FACE}
        = (\mu^{-1} \curlbold \Ebold, \curlbold \wboldh)_{0,\Omega}
        \quad \forall \wboldh\in\Vboldedgeh\\[0.5em]
        [\Piboldh \Ebold, \nabla \sh]_{\EDGE}
        = (\Ebold, \nabla \sh)_{0,\Omega}
        \quad  \forall \Ebold\in H(\curlbold,\Omega),\,\forall\sh\in\Vnodeh.
      \end{cases} 
  \end{split}
\end{equation}
The discrete bilinear form~$[\cdot,\cdot]_{\FACE}$ appearing in the
first equation of~\eqref{curl-projector-edge} is well defined.
In fact, thanks to~\eqref{exact:sequence:VEM},
$\curlbold(\Piboldh(\Ebold))$ belongs to~$\VboldfacehE$.
An analogous observation applies for the form~$[\cdot,\cdot]_{\EDGE}$
appearing in the second equation.

In order to prove the approximation properties of the projector~$\Piboldh$, we need two preliminary technical results.
\medskip 

\begin{lem} \label{lemma:inverse-estimate-face}
Consider~$\E\in\taun$. Then it exists a real parameter $1/2<s\le1$, depending on the shape regularity constant of~$\E$, such that the following inverse inequality is valid:
  \begin{equation} \label{inverse-estimate-face}
    \vert \psiboldh \vert _{s,\E}
    \lesssim \hE^{-s}
    \Vert \psiboldh \Vert_{0,\E} 
    \quad \quad \psiboldh \in \VboldfacehE.
  \end{equation}
\end{lem}
\begin{proof}
  Define~$\Psi$ as the solution to
  \begin{align*}
    \begin{cases}
      \Delta \Psi = 0                                           & \text{in } \E\\
      \nboldE \cdot \nabla \Psi = \nboldE \cdot \psiboldh       & \text{on } \partial \E\\
      \int_{\E} \Psi = 0 .
    \end{cases}
  \end{align*}
Standard regularity results for elliptic problems, see, e.g., \cite[Corollary~$23.5$]{dauge2006elliptic}, entail that there
  exists $1/2 < s \le 1$ such that
\begin{equation} \label{elliptic:regularity}
    \vert \Psi \vert_{s+1,\E} \lesssim \Vert \nboldE \cdot \psiboldh  \Vert_{s-\frac{1}{2},\partial\E}.
\end{equation}
Define~$\zbold := \psiboldh - \nabla \Psi$ and observe that~$\nboldE \cdot \zbold = 0$ on~$\partial\E$.
We recall from~\cite[Proposition~3.7]{amrouche1998vector} that for all~$\zbold \in H(\curlbold,\E) \cap H(\div,\E)$
with~$\nboldE\cdot\zbold{}_{|\partial\E} = 0$, there exists~$1/2<s\le1$ such that
\begin{equation} \label{Amrouche:prop3.7:prel}
\vert \zbold \vert_{s,\E} 
\lesssim \hE^{-s} \Vert \zbold \Vert_{0,\E} + \hE^{1 - s} \Vert \curlbold \zbold \Vert_{0,\E} + \hE^{1-s} \Vert \div \zbold \Vert_{0,\E}.
\end{equation}
Recalling that $\nboldE \cdot \zbold=0$ on $\partial\E$, $\div\zbold=\div\psiboldh \in \Pbb_0(\E)$,
and $\Psi$ has vanishing integral on $\E$, an integration by parts yields
$$
(\zbold,\nabla\Psi)_{0,\E}  = -(\div\zbold,\Psi)_{0,\E} + (\nboldE \cdot \zbold,\Psi)_{0,\partial\E} = 0 \ .
$$  
Therefore
\[
\Vert \zbold \Vert_{0,\E}^2 = (\zbold,\zbold)_{0,\E} = (\zbold,\psiboldh)_{0,\E} \le \Vert \zbold \Vert_{0,\E} \Vert \psiboldh \Vert_{0,\E} \lm{.}
\]
\lm{Thus,} \eqref{Amrouche:prop3.7:prel} becomes
\begin{equation} \label{Amrouche:prop3.7}
    \vert \zbold \vert_{s,\E} 
    \lesssim \hE^{-s} \Vert \psiboldh \Vert_{0,\E} 
    + \hE^{1 - s} \lm{(\Vert \div \zbold \Vert_{0,\E} + \Vert \curlbold \zbold \Vert_{0,\E}  )} .
  \end{equation}
We take the minimum~$s$ such that~\eqref{elliptic:regularity} and~\eqref{Amrouche:prop3.7} are valid.
Using the triangle inequality, \eqref{elliptic:regularity}, \eqref{Amrouche:prop3.7},
the fact that~$\curlbold(\nabla \Psi) = \mathbf 0$ and~$\div(\nabla \Psi)=0$,
and the fact that~$\nboldE \cdot \psiboldh{}_{|\F} \in \Pbb_0(\F)$ for all~$\F \in \EE$\lm{,}
we easily obtain that
\begin{align*}
    \begin{split}
      \vert \psiboldh \vert_{s,\E}
      & \lesssim \vert \zbold \vert_{s,\E} + \vert \Psi \vert_{s+1,\E} \\
      & \lesssim \hE^{1-s} (\Vert \div \psiboldh \Vert_{0,\E} + \Vert \curlbold \psiboldh \Vert_{0,\E})
      + \Vert \nboldE \cdot \psiboldh \Vert_{s-\frac{1}{2},\partial\E} + \hE^{-s} \Vert \psiboldh \Vert_{0,\E}.
    \end{split}
\end{align*}
We are left to show a bound for each term on the right-hand side in terms of~$\hE^{-s} \Vert \psiboldh \Vert_{0,\E}$.
We can prove such bounds based on employing polynomial inverse inequalities as, e.g., in the proofs of \cite[Proposition~$4.1$,~Proposition~$4.2$]{face-edge-VEM-interpolation-low}.
The main ingredients are the fact that~$\div \psiboldh \in \Pbb_0(\E)$, $\curlbold \psiboldh \in [\Pbb_0(\E)]^3$,
$\nboldE \cdot \psiboldh{}_{|\F} \in \Pbb_0(\F)$ for all~$\F \in \EE$, and inverse inequalities involving bubbles.
\qed
\end{proof}
\medskip

The second auxiliary result is a coercivity property on the kernel of
edge functions
\begin{align*}
  \Zboldh:= \left\{ \vboldh \in \Vboldedgeh \mid [\vboldh,\nabla \vh]_{\EDGE}=0 \; \forall\vh \in \Vnodeh   \right\}.
\end{align*}
\medskip

\begin{lem} \label{lemma:coercivity-kernel}
The following coercivity property on the kernel~$\Zboldh$ is valid:
\begin{equation} \label{coercivity:kernel}
\Vertiii{\curlbold \vboldh}_{\FACE} \gtrsim \Vert \vboldh \Vert _{0,\Omega}
\quad \quad \forall \vboldh \in \Zboldh.
\end{equation}
\end{lem}
\begin{proof}
 Given~$\vboldh \in \Zboldh$, let~$p$ be the solution to the following problem:
\begin{align*}
\begin{cases}
\text{find } p \in H^1(\Omega) \text{ such that } \int_\Omega p = 0 \text{ and} \\
(\nabla p, \nabla \xi)_{0,\Omega} = (\vboldh, \nabla \xi)_{0,\Omega}\quad \forall \xi\in H^1(\Omega).
\end{cases}
\end{align*}
  Set~$\vbold:= \vboldh - \nabla p$. We clearly have that
  \begin{equation} \label{properties:function:v}
    \div \vbold =0,\quad\quad \curlbold \vbold = \curlbold \vboldh,\quad \quad \nboldOmega \cdot \vbold _{|\partial \Omega} = 0.
  \end{equation}
For a~$1/2<s \le 1$ depending on the shape of~$\Omega$, using~\cite[Proposition~$3.7$]{amrouche1998vector} gives
\begin{equation} \label{using:Amrouche-Friedrichs}
\Vert \vbold \Vert_{s,\Omega} \lesssim \Vert \curlbold \vbold \Vert_{0,\Omega}  \overset{\eqref{properties:function:v}}{=} \Vert \curlbold \vboldh \Vert_{0,\Omega}.
\end{equation}
Denote the nodal interpolant of~$p$ by~$p_I$; see~\eqref{interpolant:nodal}.
As in Remark~\ref{remark:interpolation-exact}, the edge interpolant of~$\nabla p$
in the sense of~\eqref{interpolant:edge} is such that $(\nabla p)_I = \nabla p_I.$
Therefore, the edge interpolant of~$\vbold$ in the sense of~\eqref{interpolant:edge} satisfies
\begin{align*}
\vboldI = \vboldh - (\nabla p)_I =  \vboldh - \nabla p_I.
\end{align*}
  Next, recalling that~$\vboldh \in \Zboldh$, we observe that
  \begin{align*}
    \Vertiii{\vboldh}^2_{\EDGE} = [\vboldh, \vboldh]_{\EDGE} = [\vboldh, \vboldI + \nabla p_I]_{\EDGE}
    = [\vboldh, \vboldI]_{\EDGE} \le \Vertiii{\vboldh}_{\EDGE}  \Vertiii{\vboldI}_{\EDGE}.
  \end{align*}
  We deduce that
  \begin{equation} \label{estimate:vboldh}
    \Vertiii{\vboldh}_{\EDGE}   \le \Vertiii{\vboldI}_{\EDGE}
    \overset{\eqref{stability:edge}}{\lesssim} \Vert \vboldI \Vert_{0,\Omega}.
  \end{equation}
We estimate from above the right-hand side of~\eqref{estimate:vboldh} elementwise. Let~$\E \in \taun$.
Using the triangle inequality and~\eqref{interpolation-estimates:edge}, we write:
  \begin{equation*}
    \begin{split}
      \Vert \vboldI \Vert_{0,\E}
      & \lesssim \Vert \vbold \Vert_{0,\E} + \hE^s \vert \vbold \vert_{s,\E} 
      + \hE \Vert \curlbold \vbold \Vert_{0,\E} + \hE^{s+1} \vert \curlbold \vbold \vert_{s,\E}\\
      & \overset{\eqref{properties:function:v}}{=} \Vert \vbold \Vert_{0,\E} + \hE^s \vert \vbold \vert_{s,\E} 
      + \hE \Vert \curlbold \vboldh \Vert_{0,\E} +\hE^{s+1} \vert \curlbold \vboldh\vert_{s,\E}.\\
    \end{split}
  \end{equation*}
We know that~$\curlbold \vboldh$ belongs to~$\VboldfaceE$; see~\eqref{exact:sequence:VEM}.
Therefore, we can apply the inverse estimate~\eqref{inverse-estimate-face}, possibly taking the minimum among the scalar $s$ in \eqref{using:Amrouche-Friedrichs}
and the minimum over all elements of the parameter~$s$ appearing in Lemma~\ref{lemma:inverse-estimate-face}, and find that
  \begin{equation} \label{estimate:vboldI}
    \Vert \vboldI \Vert_{0,\E}
    \lesssim \Vert \vbold \Vert_{0,\E} + \hE^s \vert \vbold \vert_{s,\E}  + \hE \Vert \curlbold \vboldh \Vert_{0,\E}.
  \end{equation}
  Inserting~\eqref{estimate:vboldI} in~\eqref{estimate:vboldh} and
  summing over all mesh elements yield
  \begin{equation*}
    \begin{split}
      \Vert \vboldh \Vert_{0,\Omega} \overset{\eqref{stability:edge}}{\lesssim}
      \Vertiii{\vboldh}_{\EDGE} 
      \lesssim \Vert \vbold \Vert_{0,\Omega} + \vert \vbold \vert_{s,\Omega} + \Vert \curlbold \vboldh \Vert_{0,\Omega}
      \overset{\eqref{using:Amrouche-Friedrichs}}{\lesssim} \Vert \curlbold \vboldh \Vert_{0,\Omega}
      \overset{\eqref{stability:face}}{\lesssim} \Vertiii{\curlbold \vboldh} _{\FACE},
    \end{split}
  \end{equation*}
which is the assertion.
  \qed
\end{proof}
\medskip

We are in the position of proving the approximation properties of the projector~$\Piboldh$.
\medskip

\begin{prop}
Let~$\Ebold\in H^s(\curlbold,\Omega)$, $1/2<s\le1$, and $\Piboldh$ be the projector defined in~\eqref{curl-projector-edge}.
Then, the following inequality is valid:
\begin{equation} \label{approximation:Piboldh}
    \Vert \Ebold - \Piboldh \Ebold  \Vert_{\curlbold,\Omega} \lesssim 
    \h^s (\vert \Ebold \vert_{s,\Omega} + h^{1-s} \Vert \curlbold \Ebold \Vert_{0,\Omega} + \h \vert \curlbold \Ebold \vert_{s,\Omega}).
  \end{equation}
\end{prop}
\begin{proof}
Let $\EboldI$ denote the interpolant of~$\Ebold$ in the edge space~$\Vboldedgeh$; cf.~\eqref{interpolant:edge}.
Then, for a given~$\Ebold \in H(\curlbold,\Omega)$, $\Piboldh \Ebold$ is the solution to the following mixed variational problem:
  \begin{equation} \label{mixed:problem:projector}
    \begin{cases}
      [\muhat^{-1} \curlbold(\Piboldh \Ebold), \curlbold \wboldh]_{\FACE} 
      + [\wboldh, \nabla \ph]_{\EDGE}
      = (\mu^{-1} \curlbold \Ebold , \curlbold \wboldh)_{0,\Omega},\\[0.5em]
      [\Piboldh \Ebold, \nabla \sh]_{\EDGE} = (\Ebold, \nabla \sh)_{0,\Omega}
      \qquad \forall\wboldh\in \Vboldedgeh,\;\;\forall\sh\in\Vnodeh.
    \end{cases}
  \end{equation}
Indeed, it can be easily shown that~$\ph = 0$ and the coercivity of the bilinear form~$[\curlbold \cdot, \curlbold \cdot]_{\FACE}$ on the kernel~$\Zboldh$ is shown in Lemma~\ref{lemma:coercivity-kernel}.
  On the other hand, the discrete inf-sup condition for the bilinear
  form $[\cdot,\nabla\cdot]_{\EDGE}:
  \Vboldedgeh\times\Vnodeh\rightarrow\Rbb$ is a trivial consequence of
  the fact that the virtual element spaces under consideration form an
  exact sequence.
Therefore, we can use the standard analysis for mixed problems; see, e.g., \cite{BrezziFortin}.
Notably, there exist~$\wboldh$ and~$\sh$ in~$\Vboldedgeh \times \Vnodeh$ with
\begin{equation} \label{living-on-the-ball}
    \Vert \wboldh \Vert_{\curlbold,\Omega} + \Vert \sh \Vert_{1,\Omega} \le 1,
\end{equation}
such that 
  \begin{equation*}
    \begin{split}
      & \Vert \Piboldh \Ebold - \EboldI \Vert_{\curlbold,\Omega} =  \Vert \Piboldh \Ebold - \EboldI \Vert_{\curlbold,\Omega} + \Vert \ph \Vert_{1,\Omega}\\[0.5em]
      &\qquad
      \lesssim [\muhat^{-1} \curlbold(\Piboldh \Ebold - \EboldI), \curlbold \wboldh]_{\FACE}  
      + [\wboldh, \nabla \ph]_{\EDGE} 
      + [\Piboldh \Ebold - \EboldI, \nabla \sh]_{\EDGE}\\[0.5em]
      &\qquad
      \overset{\ph=0}{=}        [\muhat^{-1} \curlbold(\Piboldh \Ebold - \EboldI), \curlbold \wboldh]_{\FACE}  
      + [\Piboldh \Ebold - \EboldI, \nabla \sh]_{\EDGE} \\[0.5em]
      &\qquad
      \overset{\eqref{curl-projector-edge}}{=}
      (\mu^{-1} \curlbold \Ebold, \curlbold \wboldh)_{0,\Omega}
      - [\muhat^{-1} \curlbold \EboldI, \curlbold \wboldh]_{\FACE} 
      +(\Ebold , \nabla \sh)_{0,\Omega} - [\EboldI, \nabla \sh]_{\EDGE} =: \big[I\big] + \big[II\big].
    \end{split}
  \end{equation*}
  Since the bounds for the two terms on the right-hand side follow
  using standard VE calculations, we address them briefly.
Recall the definition of the projector~$\Pizboldh$ in~\eqref{L2-projector}. As for the term~$\big[I\big]$, we get
\begin{align*}
\big[I\big]
& \overset{\hspace{3mm}\eqref{consistency:face}\hspace{3mm}}{=} \sum_{\E \in \taun} \bigg( (\mu^{-1}(\curlbold \Ebold - \Pizboldh (\curlbold \Ebold)), \curlbold \wboldh)_{0,\E}
    -[\muhat^{-1}(\curlbold\EboldI - \Pizboldh(\curlbold\Ebold)), \curlbold \wboldh]_{\FACE,\E}
    \\[0.5em]
    & \quad\quad\quad\quad\quad\quad
    + ((\mu^{-1}-\muhat^{-1})\Pizboldh (\curlbold \Ebold) , \curlbold \wboldh)_{0,\E}
    \bigg)
    \\[0.5em]
    &
    \overset{\eqref{stability:face},\eqref{living-on-the-ball}}{\lesssim}
    \Bigg(
    \sum_{\E \in \taun}
    \bigg( \Vert \curlbold \Ebold - \Pizboldh (\curlbold\Ebold)\Vert_{0,\E} ^2
    + \Vert \curlbold(\Ebold - \EboldI) \Vert_{0,\E}^2
    \\[0.5em]
    & \quad\quad\quad\quad\quad\quad
    + \Vert \mu^{-1} - \muhat^{-1} \Vert_{L^\infty(\E)}
    \Vert \Pizboldh (\curlbold \Ebold) \Vert_{0,\E}^2
    \bigg)
    \Bigg)^{\frac12}
    \\[0.5em]
    &
    \overset{\hspace{2mm}\eqref{assumption:coefficients},\eqref{interpolation-estimates:edge}}{\lesssim}
    \h^s \vert \curlbold \Ebold \vert_{s,\Omega}
    + \h \max_{\E \in \taun}\vert \mu^{-1} \vert_{W^{1,\infty}(\E)}^2 \Vert \curlbold \Ebold \Vert_{0,\Omega}.
  \end{align*}
  We proceed similarly for the term~$\big[II\big]$:
\begin{align*}
& \big[II\big] \overset{\hspace{3mm}\eqref{consistency:edge}\hspace{3mm}}{=} \sum_{\E \in \taun} \left( (\Ebold - \Pizboldh \Ebold, \nabla \sh)_{0,\Omega} - [\EboldI - \Pizboldh \Ebold, \nabla\sh]_{\EDGE} \right) \\[0.5em]
&\overset{\eqref{stability:face},\eqref{living-on-the-ball}}{\lesssim} \left(\sum_{\E \in \taun} \left( \Vert \Ebold - \Pizboldh \Ebold \Vert_{0,\E}^2 + \Vert \Ebold - \EboldI \Vert_{0,\E}^2 \right) \right)^{\frac12} \overset{\hspace{3mm}\eqref{interpolation-estimates:edge}\hspace{3mm}}{\lesssim} \h^s\left( \vert \Ebold \vert_{s,\Omega} +\h^{1-s} \Vert\curlbold \Ebold \Vert_{0,\Omega}  + \h \vert \curlbold \Ebold \vert_{s,\Omega} \right).
\end{align*}
The assertion follows easily collecting the bounds on the terms~$\big[I\big]$, $\big[II\big]$ and by the triangle inequality.
  \qed
\end{proof}
\medskip

Next, define the weighted, global projector~$\Pboldh : [L^2(\Omega)]^3
\rightarrow \Vtildeboldfaceh$ as
\begin{equation} \label{L2-projector-face}
  [\muhat^{-1} \Pboldh \Bbold, \psiboldh]_{\FACE} =
  (\mu^{-1}\Bbold, \psiboldh)_{0,\Omega}
  \quad\quad \forall \Bbold\in [L^2(\Omega)]^3,\,\psiboldh\in\Vtildeboldfaceh.
\end{equation}
As, e.g., in~\cite{zhao2004analysis}, a crucial point in the analysis
of the semi-discrete scheme in Section~\ref{section:semi-discrete}
below is the following commuting diagram result.
\medskip

\begin{prop} \label{prop:commuting-diagram}
  Let~$\Piboldh$ and~$\Pboldh$ be the two projectors introduced
  in~\eqref{curl-projector-edge} and~\eqref{L2-projector-face},
  respectively.
  Then, the following identity is valid:
  \begin{equation} \label{commuting-diagram}
    \curlbold(\Piboldh \Ebold) = \Pboldh(\curlbold \Ebold)
    \quad\quad \forall \Ebold\in H(\curlbold, \Omega).
  \end{equation}
\end{prop}
\begin{proof}
Recall that~$\curlbold(\Piboldh\Ebold)$ belongs to~$\Vtildeboldfaceh$; see~\eqref{commuting-diagram}.  
Using~\eqref{exact:sequence:VEM}, \eqref{curl-projector-edge}, and~\eqref{L2-projector-face}, we get
\begin{align*}
    [\muhat^{-1} \curlbold(\Piboldh \Ebold), \psiboldh]_{\FACE} =
    (\mu^{-1} \curlbold \Ebold, \psiboldh)_{0,\Omega} =
    [\muhat^{-1} \Pboldh (\curlbold \Ebold), \psiboldh]_{\FACE}
    \quad \forall \psiboldh \in \Vtildeboldfaceh .
\end{align*}
The assertion follows using the stability property~\eqref{stability:face}.  \qed
\end{proof}
\medskip

In the light of commuting diagram~\eqref{commuting-diagram}, the projector~$\Pboldh$ satisfies the following property.

\medskip

\begin{lem} \label{lem:property-Pboldh}
  Let~$\Pboldh$ be defined in~\eqref{L2-projector-face}.
  Then, for all divergence free~$\psibold \in H^s(\Omega)$, $1/2<s\le 1$, the following bound is valid:
  \begin{equation} \label{property-Pboldh}
    \Vert \psibold - \Pboldh \psibold \Vert_{0,\Omega} 
    \lesssim \h \Vert \psibold \Vert_{0,\Omega} +  \h^s \vert \psibold \vert_{s,\Omega}.
  \end{equation}
\end{lem}
\begin{proof}
Let~$\Pizboldh$ be the projector defined in~\eqref{L2-projector}
and~$\psiboldI$ the interpolant in $\Vboldedgeh$ of~$\psibold$; see~\eqref{interpolant:face}.  Then, we can write
  \begin{align*}
    \Vert \psibold - \Pboldh \psibold \Vert_{0,\Omega}^2 \lesssim \Vert \psibold - \psiboldI \Vert_{0,\Omega}^2 
    + \Vert \psiboldI - \Pboldh \psibold \Vert_{0,\Omega}^2.
  \end{align*}
Since~$\div \psibold=0$, we have that $\psiboldI \in \Vtildeboldfaceh$ and thus also 
$(\psiboldI - \Pboldh \psibold) \in\Vtildeboldfaceh$; see definition~\eqref{L2-projector-face}.
We focus on the second term on the right-hand side:
\begin{equation*}
\begin{split}
\Vert \psiboldI - 
& \Pboldh \psibold \Vert_{0,\Omega}^2 \overset{\eqref{stability:face}}{\lesssim} \sum_{\E \in \taun} \Vertiii{\muhat^{-\frac12} (\psiboldI - \Pboldh \psibold)}_{\FACE,\E}^2\\
& \overset{\eqref{L2-projector-face}}{=} 
\sum_{\E \in \taun} \{ [\muhat^{-1}\psiboldI, \psiboldI - \Pboldh \psibold]_{\FACE,\E}
-(\mu^{-1}\psibold, \psiboldI-\Pboldh\psibold)_{0,\E}  \} \\
& \overset{\eqref{consistency:face}}{=}
\sum_{\E \in \taun} \{ [\muhat^{-1}(\psiboldI - \Pizboldh \psibold), \psiboldI - \Pboldh \psibold]_{\FACE,\E}
-(\mu^{-1}(\psibold - \Pizboldh \psibold), \psiboldI-\Pboldh\psibold)_{0,\E}\\
& \quad \quad \quad \quad \quad + ((\muhat^{-1}-\mu^{-1})\Pizboldh \psibold, \psiboldI - \Pboldh \psibold)_{0,\E}\} \\
& \overset{\eqref{stability:face}}{\lesssim}
      \left(\sum_{\E \in \taun} \{ \Vert \psibold - \Pizboldh \psibold \Vert_{0,\E}^2 + \Vert \psibold - \psiboldI \Vert_{0,\E}^2
      + \Vert\muhat^{-1} - \mu^{-1} \Vert_{L^{\infty}(\E)}^2 \Vert \psibold \Vert_{0,\E}^2 \}\right)^{\frac12}
      \Vert \psiboldI - \Pboldh \psibold \Vert_{0,\Omega}.
    \end{split}
  \end{equation*}
  Using~\eqref{interpolation-estimates:face} and collecting the two
  above bounds, we get the assertion.  \qed
\end{proof}

\subsection{The semi-discrete scheme}  \label{subsection:semi-discrete}
We denote the virtual element interpolant of the density current vector~$\Jbold$ in~$\Vboldedgeh$ by~$\JboldI$.
In other words, $\JboldI$ is the unique function in~$\Vboldedgeh$
satisfying~\eqref{interpolant:edge}.
Similarly, we define the interpolants~$\EboldzI$ and~$\BboldzI$ of the
initial vector field~$\Ebold^0$ and~$\Bbold^0$ in~$\Vboldedgeh$
and~$\Vboldfaceh$; cf.~\eqref{interpolant:edge} and~\eqref{interpolant:face}.
The semi-discrete scheme reads
\begin{equation}
  \label{semi-discrete-scheme}
  \begin{cases}
    \text{find } (\Eboldh, \Bboldh) \in \Vboldedgeh \times \Vboldfaceh \text{ such that }\\[0.5em]
         [\varepsilonhat \Eboldht + \sigmahat \Eboldh, \wboldh]_{\EDGE}
         - [\muhat^{-1} \Bboldh, \curlbold \wboldh]_{\FACE} =
         [\JboldI, \wboldh]_{\EDGE}
         \quad \forall \wboldh\in\Vboldedgeh \\[0.5em]
         [\muhat^{-1} \Bboldht, \psiboldh]_{\FACE}
         + [\muhat^{-1} \psiboldh, \curlbold \Eboldh]_{\FACE} = 0
         \quad \forall \psiboldh\in\Vboldfaceh,
  \end{cases}
\end{equation}
where we recall that the subscript~$t$ stands for a derivative in time.

\subsection{The fully-discrete scheme}
\label{subsection:fully-discrete}
We consider the fully-discrete approximation of~\eqref{eq:Maxwell:weak} that is obtained by applying the backward Euler time-stepping scheme to the semi-discrete problem in~\eqref{semi-discrete-scheme}.
Higher order schemes in time can be built analogously.
As customary, we start by splitting the time integration interval~$[0,T]$ in~$N$ equally spaced subintervals $[\ts{n-1},\ts{n}]$ with size~$\tau=\ts{n}-\ts{n-1}$ for all~$n=1,\dots,N$.
Moreover, let~$\Eboldh^0 =\EboldI^0\in\Vboldedgeh$
and~$\Bboldh^0=\BboldI^0\in\Vboldfaceh$ be the virtual element
interpolations of~$\Ebold^0$ and~$\Bbold^0$,
cf.~\eqref{interpolant:edge} and~\eqref{interpolant:face}, respectively
satisfying inequalities~\eqref{stability:edge}
and~\eqref{stability:face}.

Let $\Eboldh^m \in \Vboldedgeh$ and~$\Bboldh^m \in \Vboldfaceh$ be the discrete solutions at steps $m=0,\dots,N-1$.
We compute the two discrete vector fields at the time step~$\ts{m+1}$
using the implicit Euler scheme: \emph{find $\Eboldh^{m+1}\in \Vboldedgeh$
and~$\Bboldh^{m+1}\in \Vboldfaceh$ such that,
for all~$\wboldh\in\Vboldedgeh$ and~$\psiboldh\in\Vboldfaceh$,}
\begin{equation} \label{VEM:Maxwell}
  \begin{cases}
    \frac{1}{\tau} [\varepsilonhat (\Eboldh^{m+1}-\Eboldh^m), \wboldh]_{\EDGE}
    + [\sigmahat \Eboldh^{m+1}, \wboldh]_{\EDGE}
    - [\muhat^{-1} \Bboldh^{m+1}, \curlbold \wboldh]_{\FACE} = [\JboldI^{m+1},\wboldh]_{\EDGE}
    \\[1em]
    \frac{1}{\tau} [ \muhat^{-1}(\Bboldh^{m+1} -\Bboldh^m, \psiboldh)  ]_{\FACE}
    +  [\muhat^{-1} \psiboldh, \curlbold \Eboldh^{m+1}]_{\FACE} = 0.
  \end{cases}
\end{equation}
The existence and uniqueness of a solution for problem \eqref{VEM:Maxwell} follows with standard arguments.

We can simplify~\eqref{VEM:Maxwell} by rewriting the second equation as
\begin{align*}
  [\muhat^{-1}(\Bboldh^{m+1}-\Bboldh^{m} + \tau \curlbold \Eboldh^{m+1}), \psiboldh]_{\FACE} = 0
  \quad\quad \forall \psiboldh\in\Vboldfaceh.
\end{align*}
Since~$\curlbold (\Vboldedgeh)=\Vtildeboldfaceh$,
cf.~\eqref{exact:sequence:VEM}, we deduce that
\begin{equation}
  \label{wicked:identity}
  \Bboldh^{m+1} = \Bboldh^m - \tau \curlbold \Eboldh^{m+1}.
\end{equation}
Then, we substitute \eqref{wicked:identity} in the first equation of~\eqref{VEM:Maxwell} and find that
\begin{multline*}
  \hspace{2cm}
  [(\varepsilonhat+\tau \, \sigmahat) \Eboldh^{m+1}, \wboldh]_{\EDGE}
  + [\tau^2 \muhat^{-1} \curlbold \Eboldh^{m+1} , \curlbold \wboldh]_{\FACE}
  \\[0.5em]
  =[\tau \JboldI^{m+1} + \varepsilonhat \Eboldh^m, \wboldh]_{\EDGE} + [\tau \, \muhat^{-1} \Bboldh^m  , \curlbold \wboldh]_{\FACE}.
  \hspace{2cm}
\end{multline*}
This reformulation allows us to reduce the computational effort in solving~\eqref{VEM:Maxwell}.

In view of Remark~\ref{remark:interpolation-exact} and assumption~\eqref{eq:assumption:0div:initial-induction}, we find that
\begin{equation}
  \label{zero-divergence-discrete}
  \div \Bboldh^0 = \div \BboldI^0 = (\div \Bbold^0)_I=0.
\end{equation}
We use~\eqref{wicked:identity} and~\eqref{zero-divergence-discrete}, and apply the divergence operator to derive the discrete counterpart of~\eqref{eq:assumption:0div:all-times}:
\begin{align*}
  \div \Bboldh^{m} =0 \quad \quad \forall m=0,\dots,M,
\end{align*}
which implies that our scheme provides an approximation of $\Bbold$
that naturally satisfies the divergence-free constraint.

\vspace{1em}
\begin{remark}
The proposed scheme can be immediately extended to the case of general order $k > 1$ in space
by substituting the above low order spaces $\Vboldedgeh$ and $\Vboldfaceh$ with the corresponding ones from \cite{da2018family}.
The theoretical analysis would follow along the same lines as that shown below for the lowest order case.
Yet, interpolation and stability properties in high order edge and face virtual elements are work in progress.
\end{remark}


\section{Analysis of the semi-discrete scheme}
\label{section:semi-discrete}
In this section, we prove the convergence of the semi-discrete scheme~\eqref{semi-discrete-scheme}.
\medskip
\begin{thm} \label{theorem:semi-discrete:analysis}
Let~$(\Ebold,\Bbold)$ and~$(\Eboldh,\Bboldh)$ be the
solutions to~\eqref{eq:Maxwell:weak} and~\eqref{semi-discrete-scheme}
under the geometric assumptions of Section~\ref{section:meshes} and assumption~\eqref{assumption:coefficients}.
For all~$t\in[0,T]$, we assume that $\Ebold$,~$\Eboldt$ and $\Jbold$ belong to $L^1((0,T),H^s(\curlbold,\Omega))$, $1/2<s\leq1$.
Furthermore, we recall that the initial vector-valued fields~$\EboldI^0\in\Vboldedgeh$ and~$\BboldI^0\in\Vboldfaceh$ interpolate $\Ebold^0$ and~$\Bbold^0$ in the sense of~\eqref{interpolant:edge} and~\eqref{interpolant:face}, respectively.
Then, the following error estimate is valid:
\begin{equation} \label{semi-discrete:bounds}
    \Vert \Ebold(t) - \Eboldh(t) \Vert_{0, \Omega} + \Vert \Bbold(t) - \Bboldh(t) \Vert_{0,\Omega}
    \lesssim  \h^s \quad \quad \forall t \in [0,T].
\end{equation}
\end{thm}
\begin{proof}
  For all~$t\in [0,T]$, we introduce
  \begin{equation}
    \label{error-VEM-semi}
    \eboldh = \eboldh(t) :=
    \Piboldh\Ebold(t) - \Eboldh(t),
    \quad\bboldh=\bboldh(t) :=
    \Pboldh\Bbold(t) - \Bboldh(t).
  \end{equation}
We recall that $\div\BboldI^0=0$. For all~$t\in[0,T]$, the definition of quantities~\eqref{error-VEM-semi}, the projector~$\Pboldh$ in~\eqref{L2-projector-face}, and the first equation in~\eqref{semi-discrete-scheme} allow us to prove that
  \begin{equation}
    \label{equation:star:semi-discrete}
    \begin{split}
      & [\varepsilonhat \eboldht + \sigmahat \eboldh, \wboldh]_{\EDGE} - [\muhat^{-1} \bboldh, \curlbold \wboldh]_{\FACE}\\[0.5em]
      & = -[\JboldI, \wboldh]_{\EDGE} + [\varepsilonhat \Piboldh \Eboldt + \sigmahat \Piboldh \Ebold, \wboldh]_{\EDGE} - [\muhat^{-1} \Pboldh \Bbold, \curlbold \wboldh]_{\FACE}\\[0.5em]
      & = -[\JboldI, \wboldh]_{\EDGE} + [\varepsilonhat \Piboldh \Eboldt + \sigmahat \Piboldh \Ebold, \wboldh]_{\EDGE} - (\mu^{-1}  \Bbold, \curlbold \wboldh)_{0,\Omega} \quad \wboldh \in \Vboldedgeh.
    \end{split}
  \end{equation}
  Moreover, for all~$\psiboldh\in\Vtildeboldfaceh$, we
  apply~\eqref{semi-discrete-scheme},~\eqref{L2-projector-face}, the
  commuting property~\eqref{commuting-diagram},
  and~\eqref{eq:Maxwell:weak}, and obtain
  \begin{align*}
    \begin{split}
      & [\muhat^{-1} \bboldht, \psiboldh]_{\FACE} + [\muhat^{-1} \curlbold \eboldh, \psiboldh]_{\FACE}
      = [\muhat^{-1} \Pboldh \Bboldt, \psiboldh]_{\FACE} + [\muhat^{-1} \curlbold \Piboldh(\Ebold),\psiboldh]_{\FACE}\\
      & \hspace{1cm}
      = (\mu^{-1}\Bboldt,\psiboldh)_{0,\Omega} + [\muhat^{-1} \Pboldh(\curlbold \Ebold), \psiboldh]_{\FACE} = (\mu^{-1}(\Bboldt + \curlbold \Ebold), \psiboldh)_{0,\Omega} =0.
    \end{split}
  \end{align*}
  Since~$\bboldht + \curlbold \eboldh \in \Vtildeboldfaceh$, the above
  equation implies that
  \begin{equation}
    \label{an:identity:semi-discrete}
    \bboldht + \curlbold \eboldh = \mathbf 0.
  \end{equation}
  We set~$\wboldh = \eboldh$ in~\eqref{equation:star:semi-discrete},
  use~\eqref{an:identity:semi-discrete}, and deduce that
  \begin{align*}
    \begin{split}
      & [\varepsilonhat \eboldht + \sigmahat \eboldh, \eboldh]_{\EDGE} + [\muhat^{-1} \bboldh, \bboldht]_{\FACE}\\[0.5em]
      & \hspace{2cm}
      =  -[\JboldI, \eboldh]_{\EDGE}
      + [\varepsilonhat \Piboldh \Eboldt + \sigmahat \Piboldh \Ebold, \eboldh]_{\EDGE}
      - (\mu^{-1}  \Bbold, \curlbold \eboldh)_{0,\Omega}.
    \end{split}
  \end{align*}
  Next, we substitute $(\mu^{-1}\Bbold,\curlbold\eboldh)_{0,\Omega}$
  with the expression given by the first equation~\eqref{eq:Maxwell:weak}:
  \begin{equation} \label{split:T1T2T3}
    \begin{split}
      & \frac{1}{2} \partial _t \Vertiii{\varepsilonhat^{1/2}  \eboldh}^2_{\EDGE} + [\sigmahat \eboldh, \eboldh]_{\EDGE} + \frac{1}{2} \partial _t \Vertiii{\muhat^{-1/2} \bboldh}^2_{\FACE}  \\
      & \quad = \frac{1}{2} \partial _t [ \varepsilonhat \eboldh ,\eboldh]_{\EDGE} + [\sigmahat \eboldh, \eboldh]_{\EDGE} + \frac{1}{2} \partial_t [ \muhat^{-1}\bboldh, \bboldh]_{\FACE} \\
      & \quad = [-\JboldI, \eboldh] _{\EDGE} - (\muhat^{-1} \Bbold, \curlbold \eboldh)_{0,\Omega} + [\varepsilonhat \Piboldh \Eboldt + \sigmahat \Piboldh \Ebold, \eboldh]_{\EDGE}\\
      & \quad = \underbrace{(\Jbold , \eboldh)_{0,\Omega}-[\JboldI, \eboldh]_{\EDGE}}_{=:T_1}
      + \underbrace{ [\varepsilonhat \Piboldh \Eboldt, \eboldh]_{\EDGE} - (\varepsilon \Eboldt, \eboldh)_{0,\Omega} }_{=:T_2}
      + \underbrace{[\sigmahat \Piboldh \Ebold, \eboldh]_{\EDGE} - (\sigma \Ebold, \eboldh)_{0,\Omega}}_{=:T_3}.
    \end{split}
  \end{equation}
We derive an upper bound for the three terms $T_1$, $T_2$ and $T_3$ on the right-hand side of~\eqref{split:T1T2T3} separately.
To estimate the term $T_1$, we use the stability properties~\eqref{preliminary:stab} and~\eqref{stability:edge} of the bilinear form~$\SEface(\cdot, \cdot)$, employ standard polynomial approximation results, use the interpolation property~\eqref{interpolation-estimates:edge}, and obtain 
\begin{equation} \label{T1:semi-discrete}
\begin{split}
T_1 
& = (\Jbold , \eboldh)_{0,\Omega}-(\Pizboldh \JboldI, \eboldh)_{0,\Omega} - \sum_{\E \in \taun} \SEedge( (\Ibold - \Pizboldh) \JboldI, (\Ibold - \Pizboldh) \eboldh   )\\[0.5em]
& \lesssim (\Vert \Jbold - \Pizboldh \JboldI \Vert_{0,\Omega}  + \Vert (\Ibold-\Pizboldh) \JboldI \Vert_{0,\Omega}) \Vertiii{\eboldh}_{\EDGE}\\[0.5em]
& \lesssim (\Vert \Jbold - \Pizboldh \Jbold \Vert_{0,\Omega} + \Vert \Jbold - \JboldI \Vert_{0,\Omega}  + \Vert \Jbold-\Pizboldh \JboldI \Vert_{0,\Omega}) \Vertiii{\eboldh}_{\EDGE}\\[0.5em]
& \lesssim (\Vert \Jbold - \Pizboldh \Jbold \Vert_{0,\Omega} + \Vert \Jbold - \JboldI \Vert_{0,\Omega}  ) \Vertiii{\eboldh}_{\EDGE}\\[0.5em]
& \lesssim \h^s(\vert \Jbold \vert_{s,\Omega} + \h^{1-s}\Vert \curlbold \Jbold \Vert_{0,\Omega} + \h \vert \curlbold \Jbold \vert_{s,\Omega}) \Vertiii{\varepsilonhat^{1/2}  \eboldh}_{\EDGE} .
\end{split}
\end{equation}
To estimate the term~$T_2$, we introduce~$\cbold$, the piecewise constant average of~$\Eboldt$ over~$\taun$, add and subtract
$(\varepsilonhat\cbold,\eboldh)_{0,\Omega}=[\varepsilonhat\cbold,\eboldh]_{\EDGE}$,
note that $h\leq h^{s}$ for $s\leq1$, and write
  \begin{align*}
    \begin{split}
      T_2
      &
      \overset{\hspace{2mm}\eqref{consistency:edge}\hspace{3mm}}{=}
      [\varepsilonhat (\Piboldh \Eboldt - \cbold), \eboldh]_{\EDGE}
      + (\varepsilonhat (\cbold - \Eboldt), \eboldh)_{0,\Omega}
      + ((\varepsilonhat-\varepsilon) \Eboldt, \eboldh)_{0,\Omega} \\[0.5em]
      & \overset{\eqref{assumption:coefficients},\eqref{stability:edge}}{\lesssim}
      \left(   \Vert \Eboldt - \Piboldh \Eboldt \Vert _{0,\Omega} 
      + \Vert \Eboldt - \cbold \Vert_{0,\Omega}
      + \h \max_{\E\in\taun} \vert \varepsilon\vert_{W^{1,\infty}(\E)} \Vert \Eboldt \Vert_{0,\Omega} \right) \Vertiii{\eboldh}_{\EDGE}\\[0.5em]
      &
      \overset{\hspace{2mm}\eqref{approximation:Piboldh}\hspace{3mm}}{\lesssim}
      \h^s \left(
      \Vert \Eboldt \Vert_{s,\Omega} +
      \h^{1-s} \Vert \curlbold \Eboldt \Vert_{0,\Omega} +
      \h \vert \curlbold \Eboldt \vert_{s,\Omega}
      \right) \Vertiii{\varepsilonhat^{1/2} \eboldh}_{\EDGE}.
    \end{split}
\end{align*}
Recalling assumption~\eqref{assumption:coefficients} again, we treat the term~$T_3$ analogously and arrive at the bound
  \begin{equation} \label{T3:semi-discrete}
    T_3 \lesssim \h^s \left(
    \Vert \Ebold \Vert_{s,\Omega}
    + \h^{1-s} \Vert \curlbold \Ebold \Vert_{0,\Omega}
     + \h (\vert \curlbold \Ebold \vert_{s,\Omega}
    \right)
    \Vertiii{\varepsilonhat^{1/2} \eboldh}_{\EDGE}.
  \end{equation}
Introduce the regularity type term, which belong to~$L^1(0,T)$ due to the assumptions of the theorem,
\begin{align*}
\begin{split}
c_{\textbf{REG}}(t)=c_{\textbf{REG}}
& = \vert \Jbold \vert_{s,\Omega} + \Vert \Ebold  \Vert_{s,\Omega} + \Vert \Eboldt \Vert_{s,\Omega} + \h^{1-s} \Vert \curlbold \Jbold \Vert_{0,\Omega} + \h \vert \curlbold \Jbold \vert_{s,\Omega}\\
&  \quad + \h^{1-s} (\Vert \curlbold \Ebold \Vert_{0,\Omega} + \Vert \curlbold \Eboldt \Vert_{0,\Omega}) + \h (\vert \curlbold \Ebold \vert_{s,\Omega} ).
\end{split}
\end{align*}
Now, we collect the upper bounds on the terms~$T_1$, $T_2$, and~$T_3$ in~\eqref{split:T1T2T3}, recall \eqref{eq:param:bound}, and deduce that
\begin{equation} \label{bound:in-time}
\frac{1}{2} \partial_t \Vertiii{\varepsilonhat^{1/2} \eboldh}_{\EDGE}^2 \le  \frac{1}{2} \left( \partial_t \Vertiii{\varepsilonhat^{1/2} \eboldh}_{\EDGE}^2 + \partial_t \Vertiii{\muhat^{-1/2} \bboldh}_{\FACE}^2 \right) \lesssim c_{\textbf{REG}} \h^s \Vertiii{\varepsilonhat^{1/2} \eboldh}_{\EDGE}.
\end{equation}
The following identity is valid:
\begin{align*}
    \frac12 \partial_t \Vertiii{\varepsilonhat^{1/2} \,\cdot\,}_{\EDGE}^2 = \Vertiii{\varepsilonhat^{1/2} \,\cdot\,}_{\EDGE} \,\partial_t \Vertiii{\varepsilonhat^{1/2}\,\cdot\,}_{\EDGE}.
\end{align*}
We use this identity in~\eqref{bound:in-time}, so that, almost everywhere in time in $(0,T)$,
  \begin{align*}
    \partial_t \Vertiii{\varepsilonhat^{1/2} \eboldh(t)}_{\EDGE}
    \lesssim c_{\textbf{REG}} \h^s,
  \end{align*}
  and we integrate in time to obtain:
  \begin{align}
    \label{bound:time0-edge-00}
    \Vertiii{\varepsilonhat^{1/2} \eboldh(t)}_{\EDGE} \lesssim \Vertiii{ \varepsilonhat^{1/2} \eboldh(0)}_{\EDGE} 
    + \h^s \int_0^t c_{\textbf{REG}}(s) {\rm d}s 
    \lesssim \Vertiii{ \varepsilonhat^{1/2} \eboldh(0)}_{\EDGE} + \h^s .
  \end{align}
The error at the initial time $t=0$ is controlled as follows:
  \begin{equation}
    \label{bound:time0-edge}
    \begin{split}
      \Vertiii{\varepsilonhat^{1/2} \eboldh(0)}_{\EDGE}
      &
      \overset{\hspace{3mm}\eqref{stability:edge},\eqref{eq:param:bound}\hspace{3mm}}{\lesssim}
      \Vert \Ebold^0 - \EboldI^0 \Vert_{0,\Omega} 
      + \Vert \Ebold^0-\Piboldh \Ebold^0 \Vert_{0,\Omega}\\[0.5em]
      &
      \overset{\eqref{interpolation-estimates:edge}, \eqref{approximation:Piboldh}}{\lesssim}
      \h^s \left(\vert \Ebold^0 \vert_{s,\Omega} + \h^{1-s} \Vert \curlbold (\Ebold^0) \Vert_{0,\Omega}
      + \h \vert \curlbold (\Ebold^0) \vert_{s,\Omega}\right) \lesssim \h^s,
    \end{split}
  \end{equation}
  and using this inequality in~\eqref{bound:time0-edge-00} yields
  \begin{equation} \label{star1}
    \Vert \eboldh(t) \Vert_{0,\Omega} \overset{\eqref{stability:edge},\eqref{eq:param:bound}}{\lesssim}
    \Vertiii{\varepsilonhat^{1/2} \eboldh(t)}_{\EDGE}  \lesssim \h^s.
  \end{equation}
  Thus, we can write
  \begin{equation} \label{follia}
    \partial_t \Vertiii{\muhat^{-1/2} \bboldh(t)}^2_{\FACE}
    \overset{\eqref{bound:in-time}}{\lesssim} c_{\textbf{REG}} \h^s  \Vertiii{\varepsilonhat^{1/2} \eboldh(t) }_{\EDGE}
    \overset{\eqref{star1}}{\lesssim} c_{\textbf{REG}} \h^{2s}.
  \end{equation}
  Integrating in time~\eqref{follia} gives
  \begin{align*}
  \Vertiii{\muhat^{-1/2} \bboldh(t)}^2_{\FACE}
  \lesssim \Vertiii{\muhat^{-1/2} \bboldh(0)}^2_{\FACE} + h^{2s}.
  \end{align*}
  Observe that
  \begin{equation} \label{bound:time0-face}
    \Vertiii{\muhat^{-1/2}\bboldh(0)}_{\FACE}
    \overset{\eqref{stability:face},\eqref{eq:param:bound}}{\lesssim}
    \Vert \Bbold^0- \Pboldh \Bbold^0 \Vert_{0,\Omega}
    +\Vert \Bbold^0- \BboldI^0 \Vert_{0,\Omega}
    \overset{\eqref{interpolation-estimates:face},\eqref{property-Pboldh}}{\lesssim}
    \h \Vert \Bbold^0 \Vert_{0,\Omega} + \h^s \vert \Bbold^0 \vert_{s,\Omega} \lesssim \h^s.
  \end{equation}
  Then, we have
  \begin{equation} \label{star2}
    \Vert \bboldh(t) \Vert_{0,\Omega} \overset{\eqref{stability:face},\eqref{eq:param:bound}}{\lesssim}
    \Vertiii{\muhat^{-1/2}\bboldh(t)}_{\FACE} \lesssim \h^s.
  \end{equation}
Finally, we add and subtract $\Piboldh(\Ebold(t))$, $\Pboldh(\Bbold(t))$, and use the definitions of $\eboldh(t)$ and $\bboldh(t)$ and the triangle inequality to obtain
\begin{align*}
\begin{split}
& \Vert \Ebold(t) - \Eboldh(t)\Vert_{0, \Omega} + \Vert \Bbold(t) - \Bboldh(t) \Vert_{0,\Omega}\\
& \lesssim \Vert \Eboldh(t) - \Piboldh(\Ebold(t)) \Vert_{0,\Omega} +    \Vert \Bboldh(t) - \Pboldh (\Bbold(t)) \Vert_{0,\Omega} +    \Vert \eboldh(t) \Vert_{0,\Omega} + \Vert \bboldh(t) \Vert_{0,\Omega}.
\end{split}
\end{align*}
The assertion of the theorem follows on from using~\eqref{approximation:Piboldh}, \eqref{property-Pboldh},
  \eqref{star1}, and \eqref{star2}.
  \qed
\end{proof}


\section{Analysis of the fully-discrete scheme} \label{section:fully-discrete}
In this section, we prove the convergence of the fully-discrete
scheme~\eqref{VEM:Maxwell}.
Notably, we recall that we employ the implicit Euler scheme for the
time discretization and subdivide the time interval~$[0,T]$ into~$M$
sub-interval of uniform length~$\tau$.
We can extend the result below to other, possibly higher order, time discretization schemes.
\medskip

\begin{thm} \label{theorem:fully-discrete}
Let the geometric assumptions of Section~\ref{section:meshes} and assumption~\eqref{assumption:coefficients} be valid,
and~$(\Ebold,\Bbold)$ be the solutions to Maxwell's equations~\eqref{eq:Maxwell:weak}.
  We assume that~$\Ebold(t)$ and~$\Eboldt(t)$ belong
  to~$L^{\infty}((0,T), H^s(\curlbold,\Omega))$,
  and~$\partial_{tt}\Ebold$ and~$\partial_{tt}\Bbold$ to
  $L^{\infty}((0,T),[L^2(\Omega)]^3)$, $1/2<s \le 1$.
  Additionally, we assume that the electric current
  density~$\Jbold(t)$ in the right-hand side
  of~\eqref{eq:Maxwell:weak} belongs to~$L^{\infty}((0,T), H^s(\curlbold,\Omega))$ 
  for the same value of~$s$ as above.
  For all~$m=0,\dots,M$, let~$(\Eboldh^{m},\Bboldh^{m})$ denote the
  solutions of the fully-discrete scheme~\eqref{VEM:Maxwell} at the time step~$\ts{m}$.
Then, for sufficiently small $\tau$ as required in \eqref{above-bound}, the following error estimate is valid:
  \begin{equation}
    \label{bound:fully-explicit}
    \begin{split}
      & \Vert \Ebold(\ts{m}) - \Eboldhm \Vert_{0,\Omega}
      + \Vert \Bbold(\ts{m}) - \Bboldhm \Vert_{0,\Omega}
      \lesssim \h^s + \tau \quad\quad \forall m=0,\dots,M.
    \end{split}
  \end{equation}
\end{thm}
\begin{proof}
  Let~$\Piboldh$ and~$\Pboldh$ be the two projectors introduced
  in~\eqref{curl-projector-edge} and~\eqref{L2-projector-face}, whose
  approximation properties are detailed
  in~\eqref{approximation:Piboldh} and~\eqref{property-Pboldh},
  respectively, and introduce
  \begin{equation} \label{two:errors:fully-explicit}
    \eboldhm := \Piboldh(\Ebold(\tm)) - \Eboldhm,
    \quad
    \bboldhm := \Pboldh (\Bbold(\tm)) - \Bboldhm
    \quad \quad \forall m=0,\dots,M.
  \end{equation}
As a first step, we show two error equations, which we can deduce from definition~\eqref{two:errors:fully-explicit} and the fully-discrete problem~\eqref{VEM:Maxwell}.
  The first error equation reads as:
  for all~$\psiboldh \in \Vtildeboldfaceh$,
  \begin{equation} \label{first:error:equation-fully-explicit}
    \begin{split}
      & \frac{1}{\tau}[\muhat^{-1} (\bboldhmpo - \bboldhm), \psiboldh]_{\FACE}
      + [\muhat^{-1}\curlbold(\eboldhmpo), \psiboldh]_{\FACE} \\
      & \overset{\eqref{VEM:Maxwell}}{=}
      \frac{1}{\tau} [\muhat^{-1} \Pboldh(\Bbold(\tmpo) - \Bbold(\tm)), \psiboldh]_{\FACE}
      + [\muhat^{-1}\curlbold(\Piboldh \Ebold(\tmpo)), \psiboldh]_{\FACE} \\
      & \overset{\eqref{L2-projector-face}, \eqref{commuting-diagram}}{=}
      \frac{1}{\tau} (\mu^{-1} (\Bbold(\tmpo) - \Bbold(\tm)), \psiboldh)_{0,\Omega}
      + [\muhat^{-1}\Pboldh (\curlbold(\Ebold(\tmpo))), \psiboldh]_{\FACE} \\
      & \overset{\eqref{L2-projector-face}}{=}
      (\mu^{-1} ((\Bbold(\tmpo) - \Bbold(\tm))/\tau + \curlbold(\Ebold(\tmpo))), \psiboldh )_{0,\Omega}\\
      & \overset{\eqref{eq:Maxwell:strong}}{=}
      (\mu^{-1} ((\Bbold(\tmpo) - \Bbold(\tm))/\tau - \partial_t \Bbold(\tmpo)) , \psiboldh )_{0,\Omega} 
      =: (\omegam,\psiboldh   )_{0,\Omega},\\
    \end{split}
  \end{equation}
which intrinsically defines the last term $\omegam$. The second error equation reads
  \begin{equation} \label{second:error:equation-fully-explicit}
    \begin{split}
      & \frac{1}{\tau} [\varepsilonhat (\eboldhmpo-\eboldhm), \eboldhmpo]_{\EDGE}
      + [\sigmahat \eboldhmpo, \eboldhmpo]_{\EDGE} - [\muhat^{-1} \bboldhmpo, \curlbold(\eboldhmpo)]_{\FACE}\\
      & \overset{\eqref{eq:Maxwell:weak},\eqref{L2-projector-face}}{=}
      [-\JboldI, \eboldhmpo]_{\EDGE} 
      + [\varepsilonhat (\Piboldh (\Ebold(\tmpo) - \Ebold(\tm)))/\tau + \sigmahat \Piboldh (\Ebold(\tmpo)),  \eboldhmpo]_{\EDGE} \\
      &   \quad \quad \quad - (\mu^{-1} \Bbold(\tmpo), \curlbold(\eboldhmpo))_{0,\Omega}.
    \end{split}
  \end{equation}
  We pick~$\psiboldh = \bboldhmpo$ as a test function
  in~\eqref{first:error:equation-fully-explicit}, sum the resulting
  equation with~\eqref{second:error:equation-fully-explicit}, and get
  \begin{align*}
    \begin{split}
      & \frac{1}{\tau} [\varepsilonhat (\eboldhmpo-\eboldhm), \eboldhmpo]_{\EDGE}
      + [\sigmahat \eboldhmpo, \eboldhmpo]_{\EDGE} 
      + \frac1\tau[\muhat^{-1} (\bboldhmpo - \bboldhm), \bboldhmpo]_{\FACE}\\
      & = [-\JboldI, \eboldhmpo]_{\EDGE} 
      + [\varepsilonhat (\Piboldh (\Ebold(\tmpo) - \Ebold(\tm)))/\tau + \sigmahat \Piboldh (\Ebold(\tmpo)),  \eboldhmpo]_{\EDGE} \\
      &   \quad \quad \quad - (\mu^{-1} \Bbold(\tmpo), \curlbold(\eboldhmpo))_{0,\Omega} + (\omegam, \bboldhmpo)_{0,\Omega} =: T_1 + T_2 + T_3 + T_4,
    \end{split}
  \end{align*}
  where, using the first equation of~\eqref{eq:Maxwell:weak}
  with~$\eboldhmpo$ as a test function, we have set
  \begin{align}
    & T_1:= (\Jbold,\eboldhmpo)_{0,\Omega} - [\JboldI, \eboldhmpo]_{\EDGE}, \label{T1} \\[0.5em]
    & T_2:= [\varepsilonhat (\Piboldh (\Ebold(\tmpo) - \Ebold(\tm)))/\tau , \eboldhmpo]_{\EDGE} - (\varepsilon \partial_t \Ebold(\tmpo), \eboldhmpo)_{0,\Omega}, \label{T2} \\[0.5em]
    & T_3:= [\sigmahat \Piboldh(\Ebold(\tmpo)), \eboldhmpo]_{\EDGE} - (\sigma \Ebold(\tmpo), \eboldhmpo)_{0,\Omega}, \label{T3}\\[0.5em]
    & T_4:= (\mu^{-1} (\Bbold(\tmpo) - \Bbold(\tm))/\tau - \partial_t \Bbold(\tmpo),\bboldhmpo)_{0,\Omega}.  \label{T4}
  \end{align}
  We easily deduce that
\begin{equation} \label{first:bound:fully-explicit}
\begin{split}
& \Vertiii{\varepsilonhat^{1/2}  \eboldhmpo}_{\EDGE}^2 + \Vertiii{\muhat^{1/2} \bboldhmpo}_{\FACE}^2 \\
& 
 \lesssim \Vertiii{\varepsilonhat^{1/2}\eboldhm}_{\EDGE} \Vertiii{\varepsilonhat^{1/2}\eboldhmpo}_{\EDGE}   + \Vertiii{\muhat^{1/2}\bboldhm}_{\FACE} \Vertiii{\muhat^{1/2}\bboldhmpo}_{\FACE} + \tau (T_1 + T_2 + T_3 + T_4).
\end{split}
\end{equation}
For the time being, assume the following bound is valid:
\begin{equation} \label{bound:T1-T4}
\tau(T_1 + T_2 + T_3 + T_4) \lesssim \tau (\h^s + \tau) \Vertiii{\varepsilonhat^{1/2}\eboldhmpo}_{\EDGE}
\lesssim \tau \Vertiii{\varepsilonhat^{1/2}\eboldhmpo}^2_{\EDGE} + \tau (\h^s + \tau)^2.
\end{equation}
We shall show~\eqref{bound:T1-T4} at the end of the proof.
Inserting~\eqref{bound:T1-T4} in~\eqref{first:bound:fully-explicit}
and some standard manipulations yield, for a positive~$c$ independent of~$\h$ and~$\tau$,
\begin{align*}
\begin{split}
& \Vertiii{\varepsilonhat^{1/2}\eboldhmpo}_{\EDGE}^2 + \Vertiii{\muhat^{1/2}\bboldhmpo}_{\FACE}^2\\
& \le c\left( \Vertiii{\varepsilonhat^{1/2}\eboldhm}_{\EDGE}^2 + \Vertiii{\muhat^{1/2}\bboldhm}_{\FACE}^2
 + \tau \Vertiii{\varepsilonhat^{1/2}\eboldhmpo}^2_{\EDGE} + \tau(\h^s + \tau)^2 \right).
\end{split}
\end{align*}
In other words, for $\tau \le 1/c$, we have
\begin{equation} \label{above-bound}
 \Vertiii{\varepsilonhat^{1/2}\eboldhmpo}_{\EDGE}^2 + \Vertiii{\muhat^{1/2}\bboldhmpo}_{\FACE}^2
 \le \frac{1}{1-c\tau} (\Vertiii{\varepsilonhat^{1/2}\eboldhm}_{\EDGE}^2 + \Vertiii{\muhat^{1/2}\bboldhm}_{\FACE}^2)+ \frac{c\tau}{1-c\tau} (\h^s + \tau)^2.
\end{equation}
Bound~\eqref{above-bound} has the form
\begin{align*}
  a_{m+1} \le \frac{1}{1-c\tau} a_m + \frac{c\tau}{1-c\tau} (\h^s + \tau)^2.
\end{align*}
Recalling that~$M:=T/\tau$, we can iterate the above bound and write
\begin{align*}
  \begin{split}
    a_m 
    & \le \left( \frac{1}{1-c\tau}  \right)^M a_0 + c\tau  \left(\sum_{i=1}^M \left( \frac{1}{1 - c\tau} \right)^i \right) (\h^s + \tau)^2\\
    & \le \left( \frac{1}{1-c\tau}  \right)^M a_0 + c T \left( \frac{1}{1-c\tau} \right)^M  (\h^s + \tau)^2 = \left( \frac{1}{1-c\tau} \right)^M \left( a_0 + c T (\h^s + \tau)^2 \right)
    \quad \quad  \forall m=1,\dots,M.
  \end{split}
\end{align*}
By noting that
\begin{align*}
\left( \frac{1}{1-c\tau} \right)^{M} =\left( \frac{1}{1-c\tau} \right)^{\frac{T}{\tau}}
\end{align*}
is uniformly bounded as~$\tau \rightarrow 0$, we achieve
\begin{align*}
  a_m \lesssim a_0 + (\h^s + \tau)^2 \quad \quad  \forall m=1,\dots,M.
\end{align*}
Thus, bound~\eqref{above-bound} yields
\begin{equation}
  \Vertiii{\varepsilonhat^{1/2}\eboldhmpo}_{\EDGE}^2 + \Vertiii{\muhat^{1/2}\bboldhmpo}_{\FACE}^2
  \lesssim \Vertiii{\varepsilonhat^{1/2}\eboldh^0}_{\EDGE}^2 + \Vertiii{\muhat^{1/2}\bboldh^0}_{\FACE}^2 + (\h^s + \tau)^2.
\end{equation}
The assertion of the theorem follows from the triangle inequality, the data assumptions \eqref{eq:param:bound},
stability properties~\eqref{stability:edge}
and~\eqref{stability:face}, and bounds~\eqref{bound:time0-edge}
and~\eqref{bound:time0-face} on the initial data approximation error.

\medskip
Thus, we are left to show~\eqref{bound:T1-T4}, i.e., the upper bounds
on the terms~$T_i$, $i=1,\dots,4$, in~\eqref{T1}-\eqref{T4}.
In the following bounds, we shall use the data assumption~\eqref{eq:param:bound} several times.
Therefore, this will not be declared at every instance.
We can deal with the terms~$T_1$ and~$T_3$ as in the semi-discrete
analysis of Theorem~\ref{theorem:semi-discrete:analysis}. 
More precisely, proceeding as in~\eqref{T1:semi-discrete}
and~\eqref{T3:semi-discrete}, we can write
\begin{align*}
  T_1 \lesssim \h^s (\vert \Jbold \vert_{s,\Omega}
  + \h^{1-s}\Vert \curlbold \Jbold \Vert_{0,\Omega}
  + \h      \vert \curlbold \Jbold \vert_{s,\Omega}) \Vertiii{\eboldhmpo}_{\EDGE}
  \lesssim \h^s \Vertiii{\varepsilonhat^{1/2}\eboldhmpo}_{\EDGE}
\end{align*}
and
\begin{align*}
  \begin{split}
    T_3
    & \overset{\eqref{assumption:coefficients}}{\lesssim} \h^s (\vert \Ebold (\tmpo) \vert_{s,\Omega} + \h^{1-s} \Vert \curlbold(\Ebold(\tmpo))\Vert_{0,\Omega}
    + \h \max_{\E \in \taun} \vert \varepsilon \vert_{W^{1,\infty}(\E)}  \vert \curlbold(\Ebold(\tmpo)) \vert_{s,\Omega})\Vertiii{\eboldhmpo}_{\EDGE}\\
    & \lesssim \h^s \Vertiii{\varepsilonhat^{1/2}\eboldhmpo}_{\EDGE}.
  \end{split}
\end{align*}
Next, we focus on the two remaining terms and start with~$T_4$:
\begin{align*}
  \begin{split}
    T_4 
    & := (\mu^{-1} \partial_t \Bbold(\tmpo) - (\Bbold(\tmpo) - \Bbold(\tm)) /\tau, \eboldhmpo)_{0,\Omega} \\
    & \overset{\eqref{stability:edge}}{\lesssim} \Vert \partial_t \Bbold(\tmpo) - (\Bbold(\tmpo) - \Bbold(\tm)) /\tau \Vert_{0,\Omega} 
    \Vertiii{\eboldhmpo}_{\EDGE}\\
    & = \Vert \frac{1}{\tau} \int_{\tm}^{\tmpo} (s-\tm) \partial_{tt}\Bbold(s) ds \Vert_{0,\Omega} \Vertiii{\eboldhmpo}_{\EDGE}\\
    & \le \tau \Vert \partial_{tt}\Bbold \Vert_{L^{\infty}([\tm,\tmpo],L^2(\Omega))} \Vertiii{\eboldhmpo}_{\EDGE} \lesssim \tau \Vertiii{\varepsilonhat^{1/2}\eboldhmpo}_{\EDGE}.\\
  \end{split}
\end{align*}
As for the term~$T_2$, we consider the splitting
\begin{align*}
  \begin{split}
    T_2
    & = [\varepsilonhat(\Piboldh(\Ebold(\tmpo)-\Ebold(\tm)))/\tau, \eboldhmpo]_{\EDGE} - (\varepsilon (\Ebold(\tmpo)-\Ebold(\tm))/\tau ,  \eboldhmpo)_{0,\Omega} \\
    & \quad +  (\varepsilon (\Ebold(\tmpo)-\Ebold(\tm))/\tau,  \eboldhmpo)_{0,\Omega} - (\varepsilon \partial_t \Ebold(\tmpo), \eboldhmpo)_{0,\Omega}
    =: T_{2,1} + T_{2,2}.
  \end{split}
\end{align*}
The term~$T_{2,2}$ is dealt with as the term~$T_4$:
\begin{align*}
  T_{2,2}\le \tau \Vert \partial_{tt} \Ebold \Vert_{L^{\infty}([\tm,\tmpo], L^2(\Omega))} \Vertiii{\eboldhmpo} \lesssim \tau \Vertiii{\varepsilonhat^{1/2}\eboldhmpo}_{\EDGE}.
\end{align*}
Finally, we show an upper bound on the term~$T_{2,1}$ proceeding as for the term~$T_3$:
\begin{align*}
  T_{2,1} \lesssim \h^s \Vert (\Ebold(\tmpo) - \Ebold(\tm))/\tau \Vert_{*,\Omega} \Vertiii{\varepsilonhat^{1/2}\eboldhmpo}_{\EDGE},
\end{align*}
where we have used~\eqref{assumption:coefficients} again and set
\begin{align*}
  \Vert \cdot \Vert_{*,\Omega} := \vert \cdot \vert_{s,\Omega} + \h^{1-s} \Vert \curlbold(\cdot) \Vert_{0,\Omega} + \h \vert \curlbold(\cdot) \vert_{s,\Omega}.
\end{align*}
We must prove that the *-norm of the difference quotient is finite. To
this aim, observe
\begin{align*}
  \begin{split}
    & \Vert (\Ebold(\tmpo)-\Ebold(\tm)/\tau) \Vert_{*,\Omega}
    = \left\Vert 1/\tau \int_{\tm}^{\tmpo} \partial_t \Ebold(s) ds \right\Vert_{*,\Omega}\\
    & \le 1/\tau \int_{\tm}^{\tmpo} \Vert \partial_t \Ebold(s)  \Vert_{*,\Omega} ds 
    \le \Vert \partial_t \Ebold  \Vert_{L^{\infty}((0,T), H^s(\curlbold,\Omega))}.
  \end{split}
\end{align*}
Collecting the bounds on the terms~$T_1$, $T_{2,1}$, $T_{2,2}$, $T_{3}$,
and~$T_{4}$, we deduce~\eqref{bound:T1-T4}, whence the assertion
follows.
\qed
\end{proof}


\section{Numerical results}
\label{section:numerical:results}

In this section, we investigate the accuracy of the fully discrete scheme~\eqref{VEM:Maxwell}.
To this end, we consider three different mesh families:
\medskip
\begin{itemize}
\item\texttt{cube}: regular cubic meshes;
\item\texttt{voro}: Voronoi tessellations optimized by the Lloyd algorithm;
\item\texttt{rand}: Voronoi tessellations of a cloud of points that are randomly positioned in the computational domain.
\end{itemize}
\medskip

We selected these three types of meshes as they offer an increasing level of geometric difficulty.
Indeed, the meshes in \texttt{cube} are uniform;
the meshes in \texttt{voro} may have small edges and faces but the geometric shape of the mesh elements is not distorted; finally meshes \texttt{rand} may have small edges and faces, as well as stretched polyhedral elements.
We refer to a specific partition of $\Omega$ by the corresponding keyword (\texttt{cube}, \texttt{voro}, and \texttt{rand}) followed by the number of elements.
For example, ``\texttt{voro125}'' refers to a mesh made of~$125$ Voronoi cells optimized by the Lloyd algorithm.

We numerically verify the optimal convergence rate in the $L^2$ norm of the approximation to the electric field~$\Ebold$ and the magnetic
flux field~$\Bbold$ on a sequence of four refined meshes for each mesh family.
These four meshes have a decreasing mesh size. 
We show the third mesh of each family in Figure~\ref{fig:meshes}.
\begin{figure}[!htb]
  \centering
  \begin{tabular}{ccc}
    \texttt{cube1000} &\texttt{voro1000} &\texttt{rand1000} \\
    \includegraphics[width=0.30\textwidth]{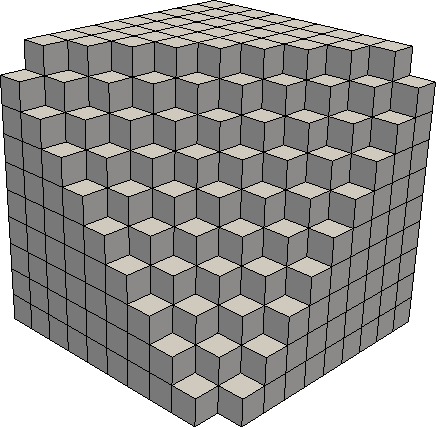} &\phantom{mm}
    \includegraphics[width=0.30\textwidth]{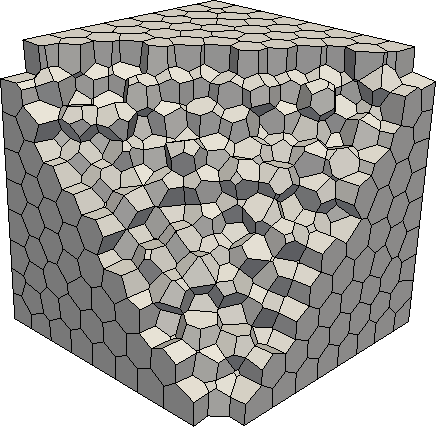} &\phantom{mm}
    \includegraphics[width=0.30\textwidth]{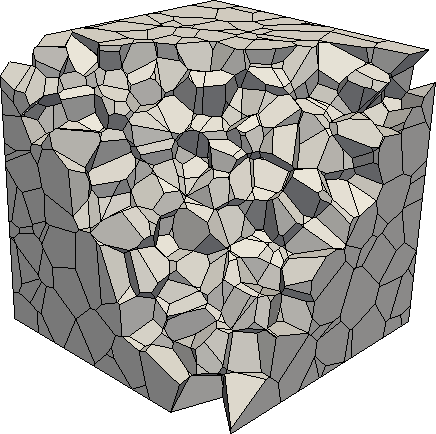} 
  \end{tabular}
  \caption{The third mesh of the mesh families \texttt{cube} (left panel), \texttt{voro} (middle panel), and \texttt{rand} (right panel).}
  \label{fig:meshes}
\end{figure}
The virtual element approximations~$\Ebold_h$ and~$\Bbold_h$ to~$\Ebold$ and~$\Bbold$ are not available in closed form
so we evaluate the error in the~$L^2$ norm at any time $T$ by using the polynomial projections~$\Pizboldh \Ebold_h$ and~$\Pizboldh \Bbold_h$.

We performed a sensitivity analysis on the stabilizations terms that appear in the definition of edge and face discrete scalar products.
More precisely, we inserted two constant coefficients, $\eta_{\text{edge}}^K$ and $\eta_{\text{face}}^K$ in front of $\SEedge(\cdot,\cdot)$ and $\SEface(\cdot,\cdot)$, respectively.
We observe that the best choice is given by $\eta_{\text{edge}}^K=0.01$ and $\eta_{\text{face}}^K=0.5$ and use these values in all numerical tests.

\subsection{Test Case 1: constant coefficients}
We solve Maxwell's equations on the computational domain~$\Omega=(0,\,1)^3$ for~$t\in[0,1]$
with constant coefficients $\varepsilon,\,\sigma$, and $\mu$ equal to~$1$.
The boundary condition and the current density vector are computed by taking the exact solution fields
\begin{equation*}
  \Ebold(t,x,y,z) = t\curlbold\phibold(x,y,z) + t^2\psibold(x,y,z)
  \quad\text{and}\quad
  \Bbold(t,x,y,z) = \frac{t^2}{2} \curlbold\curlbold\phibold(x,y,z)\, ,
\end{equation*}
where the auxiliary vector-valued fields~$\phibold$ and~$\psibold$ are defined as
\[
\phibold = \left(\begin{array}{c}
    \sin^2(\pi x)\, y^2(1-y)^2    \, z^2(1-z)^2     \\[0.5em]
    x^2(1-x)^2   \, \sin^2(\pi y) \, z^2(1-z)^2     \\[0.5em]
    x^2(1-x)^2   \, y^2(1-y)^2    \, \sin^2(\pi z)  \end{array}\right),
  \quad \quad \psibold = \nabla\big(\sin(\pi x))\sin(\pi y)\sin(\pi z)\big).
\]
A straightforward calculation shows that $\div\Bbold=0$, i.e., the magnetic field $\Bbold$ is solenoidal.

In Figure~\ref{fig:convEBTimeAndSpace1}, we plot the $L^2$ errors at final time~$T=1$ for simultaneous refinements of $h$ and $\tau$ on the three mesh families
and observe the expected convergence rate, which we recall has to be proportional to~$h+\tau$;
see Theorem~\ref{theorem:fully-discrete}.
\begin{figure}[!htb]
\centering
\begin{tabular}{cc}
\includegraphics[width=0.45\textwidth]{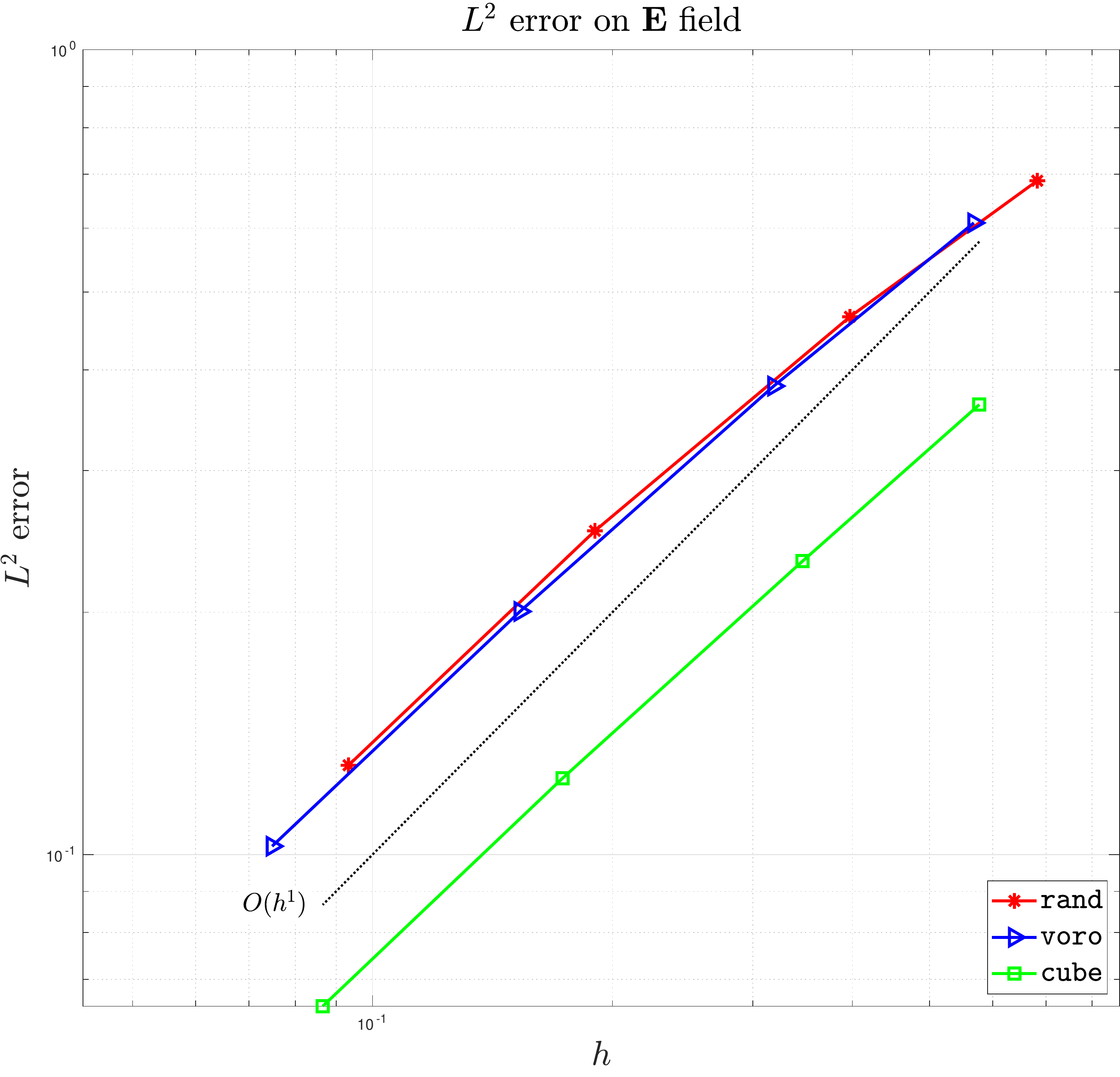} &
\includegraphics[width=0.45\textwidth]{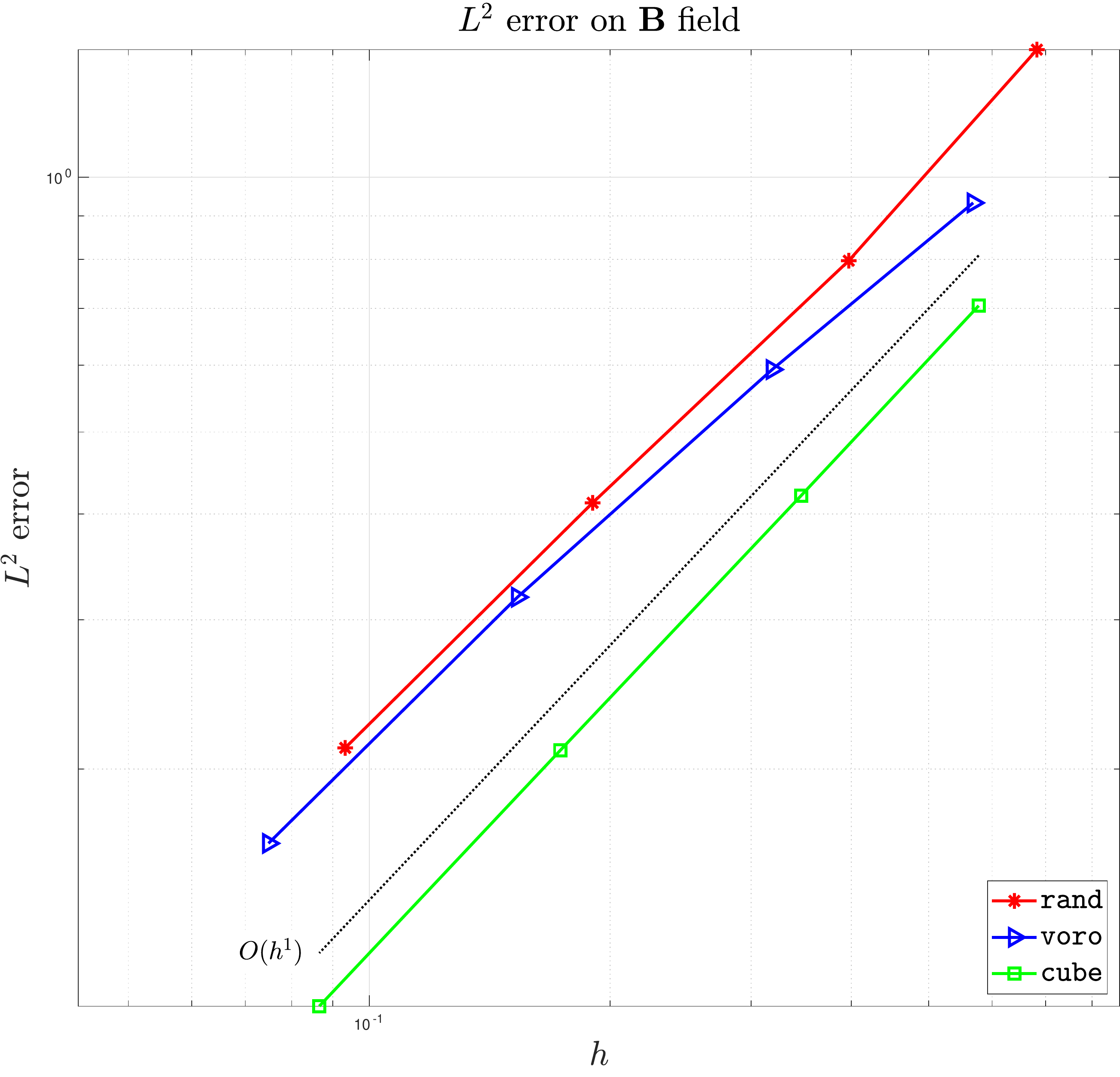} \\
\end{tabular}
\caption{Test Case~1: Error curves in the $L^2$ norm at final time~$T=1$ for the virtual element approximation~$\Ebold_h$ (left panel) and~$\Bbold_h$ (right panel)
using mesh families \texttt{cube}, \texttt{voro}, and \texttt{rand} for simultaneous refinements of $h$ and $\tau$.}
\label{fig:convEBTimeAndSpace1}
\end{figure}

In Tables~\ref{tab:randEConvIE1} and~\ref{tab:randBConvIE1}
we report the approximation errors for \texttt{rand} meshes only since we observe similar behavior of both \texttt{cube} and \texttt{voro} families.

Each column of the tables shows how the method converges with respect to space discretization, i.e., by using a constant time step and refining the mesh.
Likewise, each row of the tables shows how the method converges with respect to the time discretization, i.e., by halving the time step on a fixed mesh.
The error in space seems to be the dominant effect so it hides the convergence in time.
Indeed, the errors does not halve along rows while they do halve along colums.
However, the errors along the diagonal show how VEM behaves when we simultaneously refine the numerical calculations in space and time.

\begin{table}[!htb]
\centering
\begin{tabular}{|r|c|c|c|c|c|}
    \hline
    \multicolumn{1}{|c|}{$h/\tau$}&1 &1/2 &1/4 &1/8 &1/16\\
    \hline
    \texttt{rand27}   & 7.59493e-01 $\:$& 6.87305e-01 $\:$& 6.61804e-01 $\:$& 6.53871e-01 $\:$& 6.51503e-01 $\:$\\ 
    \texttt{rand125}  & 6.75411e-01 $\:$& 5.28227e-01 $\:$& 4.65715e-01 $\:$& 4.42600e-01 $\:$& 4.34238e-01 $\:$\\ 
    \texttt{rand1000} & 5.57938e-01 $\:$& 3.71547e-01 $\:$& 2.84666e-01 $\:$& 2.52453e-01 $\:$& 2.42061e-01 $\:$\\ 
    \texttt{rand8000} & 5.15998e-01 $\:$& 3.05485e-01 $\:$& 1.93696e-01 $\:$& 1.45600e-01 $\:$& 1.29105e-01 $\:$\\ 
    \hline
\end{tabular}
\caption{Test Case~1: $L^2$-norms at final time~$T=1$ of the error for the virtual element approximation $\Ebold_h$ using mesh family~\texttt{rand} for various combinations of $h$ and $\tau$.}
\label{tab:randEConvIE1}
\end{table}
\begin{table}[!htb]
  \centering
  \begin{tabular}{|r|c|c|c|c|c|}
    \hline
    \multicolumn{1}{|c|}{$h/\tau$}&1 &1/2 &1/4 &1/8 &1/16\\
    \hline
    \texttt{rand27}   & 1.44376e+00 $\:$& 1.41488e+00 $\:$& 1.40800e+00 $\:$& 1.41438e+00 $\:$& 1.42469e+00 $\:$\\ 
    \texttt{rand125}  & 8.04334e-01 $\:$& 7.98555e-01 $\:$& 7.96968e-01 $\:$& 7.96871e-01 $\:$& 7.96838e-01 $\:$\\ 
    \texttt{rand1000} & 4.17143e-01 $\:$& 4.13694e-01 $\:$& 4.12693e-01 $\:$& 4.12505e-01 $\:$& 4.12469e-01 $\:$\\ 
    \texttt{rand8000} & 2.15819e-01 $\:$& 2.12642e-01 $\:$& 2.11953e-01 $\:$& 2.11841e-01 $\:$& 2.11817e-01 $\:$\\
    \hline
\end{tabular}
\caption{Test Case~1: $L^2$-norms of the error at final time~$T=1$ for the virtual element approximation $\Bbold_h$
using mesh family~\texttt{rand} for various combinations of $h$ and $\tau$.}
\label{tab:randBConvIE1}
\end{table}

Finally, in Table~\ref{tab:voroDivBConvIE1} we report the $L^2$-norm of the divergence of~$\Bbold_h$ for each combination of~$h$ and~$\tau$.
This table confirms that the numerical approximation to the magnetic field provided by the VEM is divergence free.
Indeed, all the values of the divergence are very small even if a slight growth is visible during the refinement process for $h\to0$,
which is very likely due to round-off effects related to the conditioning of the final linear system.

Such interpretation is also supported by the results presented in~\cite{alvarez2020virtual}.
Here, it was noted that the $L^2$-norm of the $\div \Bbold_h$ may be affected by the residual threshold at which the iterations of a preconditioned Krilov method are arrested.
More precisely, the authors of~\cite{alvarez2020virtual} noted that the higher this threshold is, the bigger the $L^2$ norm of $\div\Bbold_h$ is.
Consequently, we can infer that the divergence-free property of~$\Bbold_h$ is related to \emph{how well} the linear system is solved and we claim that
this effect on the $L^2$-norm of $\div\Bbold_h$ is probably due a possible growth of the condition number of the final linear system.
We use the direct solver PARDISO~\cite{pardiso}.
Thus, the divergence free condition is not affected by any parameters of the solver;
rather, it is related \emph{only} to the round-off error.
\begin{table}[!htb]
  \centering
  \begin{tabular}{|r|c|c|c|c|c|}
    \hline
    \multicolumn{1}{|c|}{$h/\tau$}&1 &1/2 &1/4 &1/8 &1/16\\
    \hline
    \texttt{rand27}   & 5.57005e-14 $\:$& 2.33412e-14 $\:$& 1.43632e-14 $\:$& 1.31831e-14 $\:$& 1.01836e-14 $\:$\\ 
    \texttt{rand125}  & 4.47399e-13 $\:$& 3.14928e-13 $\:$& 1.27115e-13 $\:$& 1.26818e-13 $\:$& 1.07065e-13 $\:$\\ 
    \texttt{rand1000} & 5.41757e-12 $\:$& 2.49285e-12 $\:$& 1.74853e-12 $\:$& 1.17031e-12 $\:$& 9.19428e-13 $\:$\\ 
    \texttt{rand8000} & 7.24382e-10 $\:$& 3.80017e-10 $\:$& 6.07471e-10 $\:$& 2.25296e-10 $\:$& 3.78838e-09 $\:$\\ 
    \hline
  \end{tabular}
\caption{Test Case~1: $L^2$-norm of $\div \Bbold_h$ at final time~$T=1$ using mesh family~\texttt{rand} for various combinations of $h$ and $\tau$.}
\label{tab:voroDivBConvIE1}
\end{table}

\subsection{Test case 2: polarized fields with variable coefficients}
We solve Maxwell's equations on the computational domain~$\Omega=(0,\,1)^3$ for~$t\in[0,1]$ with the variable coefficients
\begin{equation}
\mu(x,\,y,\,z)      := \frac{1}{1+x^2+y^2+z^2},\quad \varepsilon(x,\,y,\,z) := 2-x^2-z, \quad  \sigma(x,\,y,\,z)  := 2-y^2+z. \label{eqn:coeff}
\end{equation}
The boundary conditions and the current density vector~$\Jbold$ are defined in accordance with~\eqref{eqn:coeff} and the exact solution fields
\[
\Ebold(\xbold,\,t):=\left(\begin{array}{c} 0 \\ 0 \\ \sin(\pi\,x)\,\sin(\pi\,y) \end{array}\right)\cos(2.2\,\pi\,t),
\quad \quad \Bbold(\xbold,\,t):=   \left(\begin{array}{c} -\cos(\pi\,y)\,\sin(\pi\,x) \\ \phantom{-}\cos(\pi\,x)\,\sin(\pi\,y)\\ 0   \end{array}\right)\sin(2.2\,\pi\,t)/2.2  .
\]
The electromagnetic fields~$\Ebold$ and~$\Bbold$ are orthogonal at any point in~$\Omega$ and time in~$[0,1]$.
Consequently, this solution simulates a polarized stationary electromagnetic wave with a polarization direction that is parallel to $\Ebold\times\Bbold$.
We underline that this second test case is more complex than the previous one since the coefficients $\mu$, $\epsilon$ and $\sigma$ are all variables in space.

In~Figure~\ref{fig:convEBTimeAndSpace2},
we plot the $L^2$ errors at final time~$T= 1$ for simultaneous refinements of~$h$ and~$\tau$ on the three mesh families and observe the expected convergence rate,
which we recall is expected to be proportional to $h+\tau$; see Theorem~\ref{theorem:fully-discrete}.
Moreover, in Figure~\ref{fig:convEBDifferentTime} we report the convergence rate at time $T=0.25, 0.50$ and 0.75.
Also in this case the behaviour of the error is the one predicted by Theorem~\ref{theorem:fully-discrete}. 
We show such convergence lines only for \texttt{rand} meshes as the results for the other type of meshes are similar.

\begin{figure}[!htb]
  \centering
  \begin{tabular}{cc}
    \includegraphics[width=0.45\textwidth]{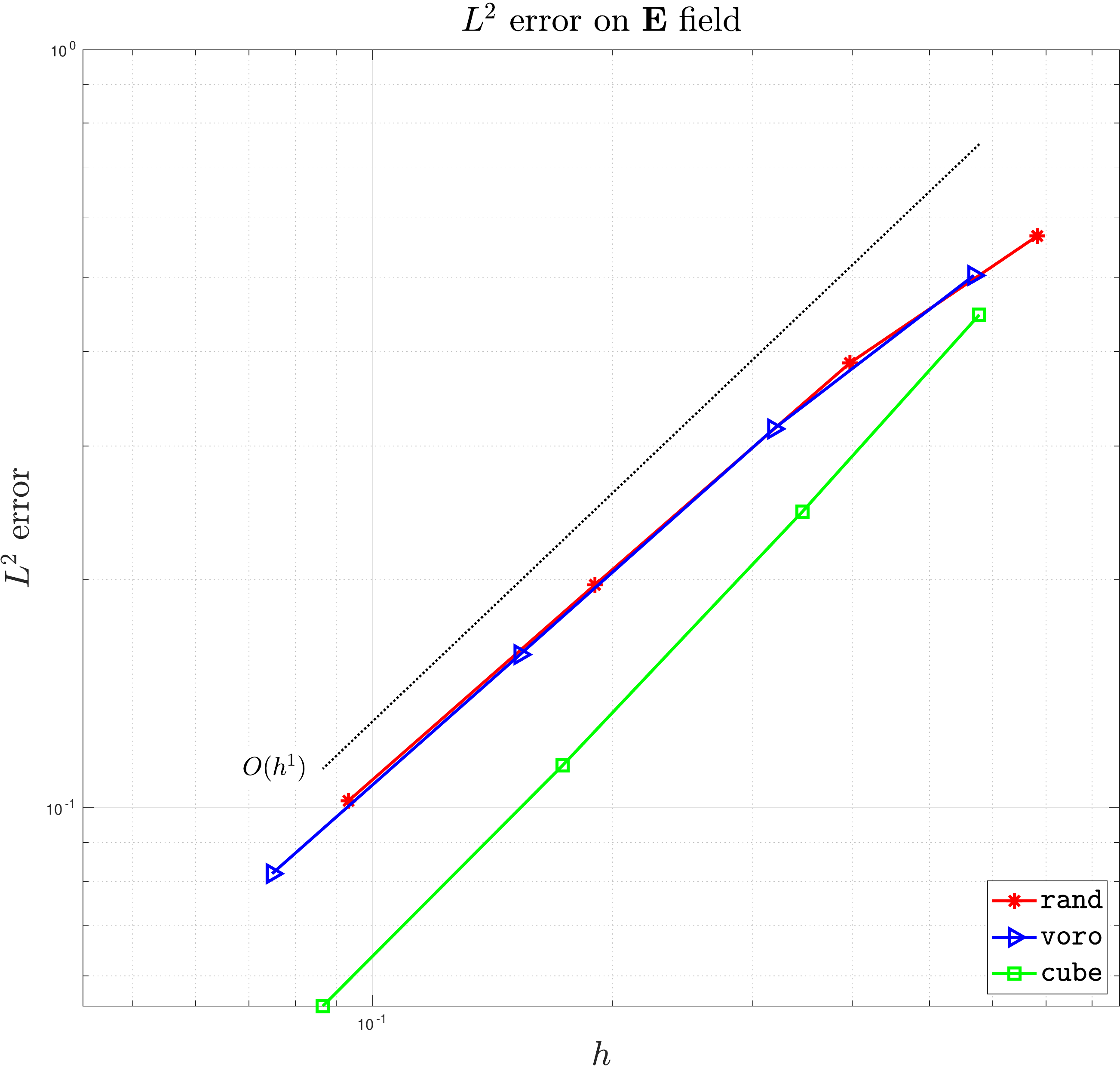} &
    \includegraphics[width=0.45\textwidth]{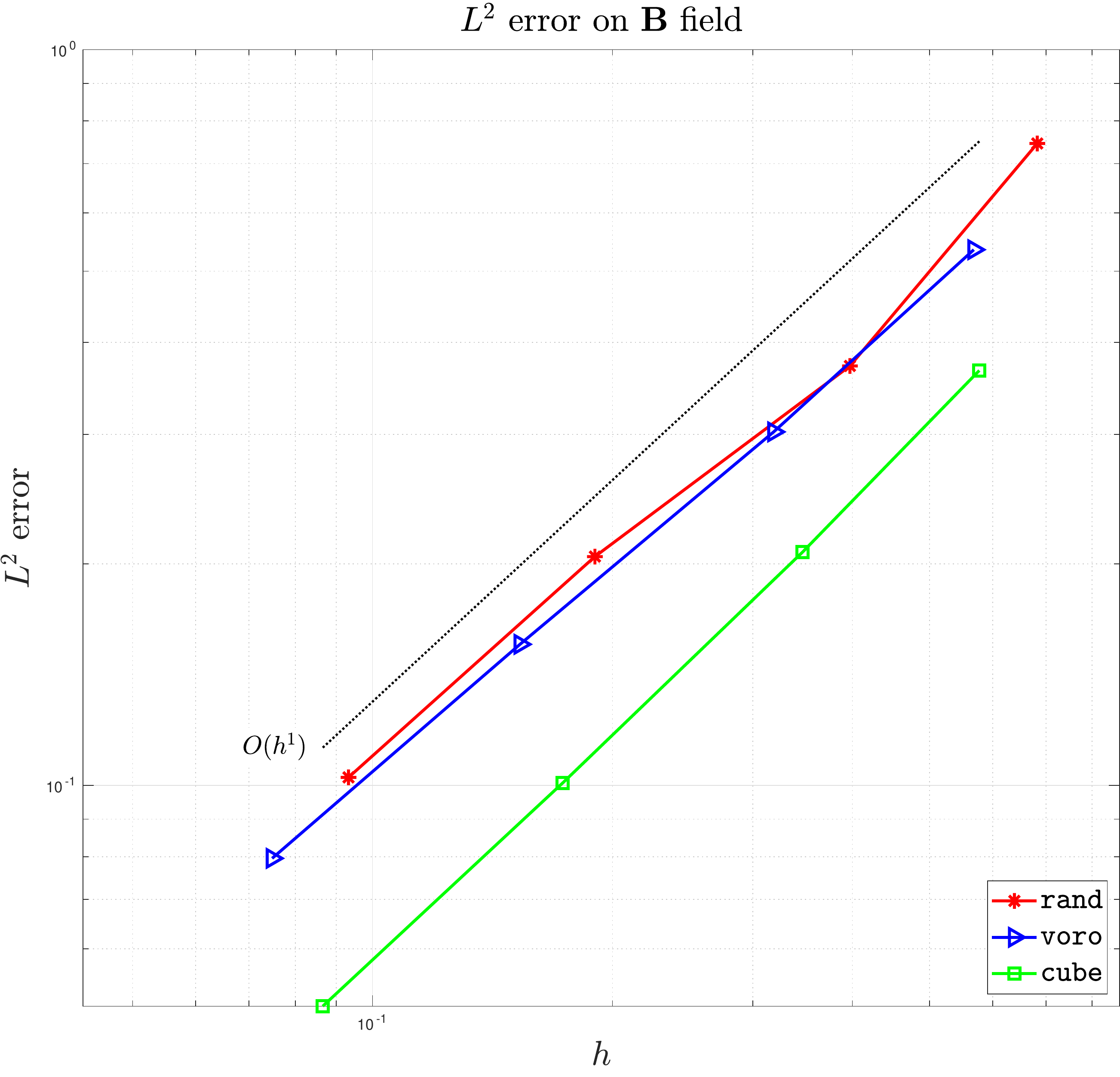} \\
  \end{tabular}
\caption{Test Case~2: error curves in the $L^2$-norm at final time~$T=1$ for the virtual  element approximation $\Ebold_h$ (left panel) and $\Bbold_h$ (right panel)
using mesh families \texttt{cube}, \texttt{voro}, and \texttt{rand} for simultaneous refinement of~$h$ and~$\tau$.}
\label{fig:convEBTimeAndSpace2}
\end{figure}

\begin{figure}[!htb]
  \centering
  \begin{tabular}{cc}
    \includegraphics[width=0.45\textwidth]{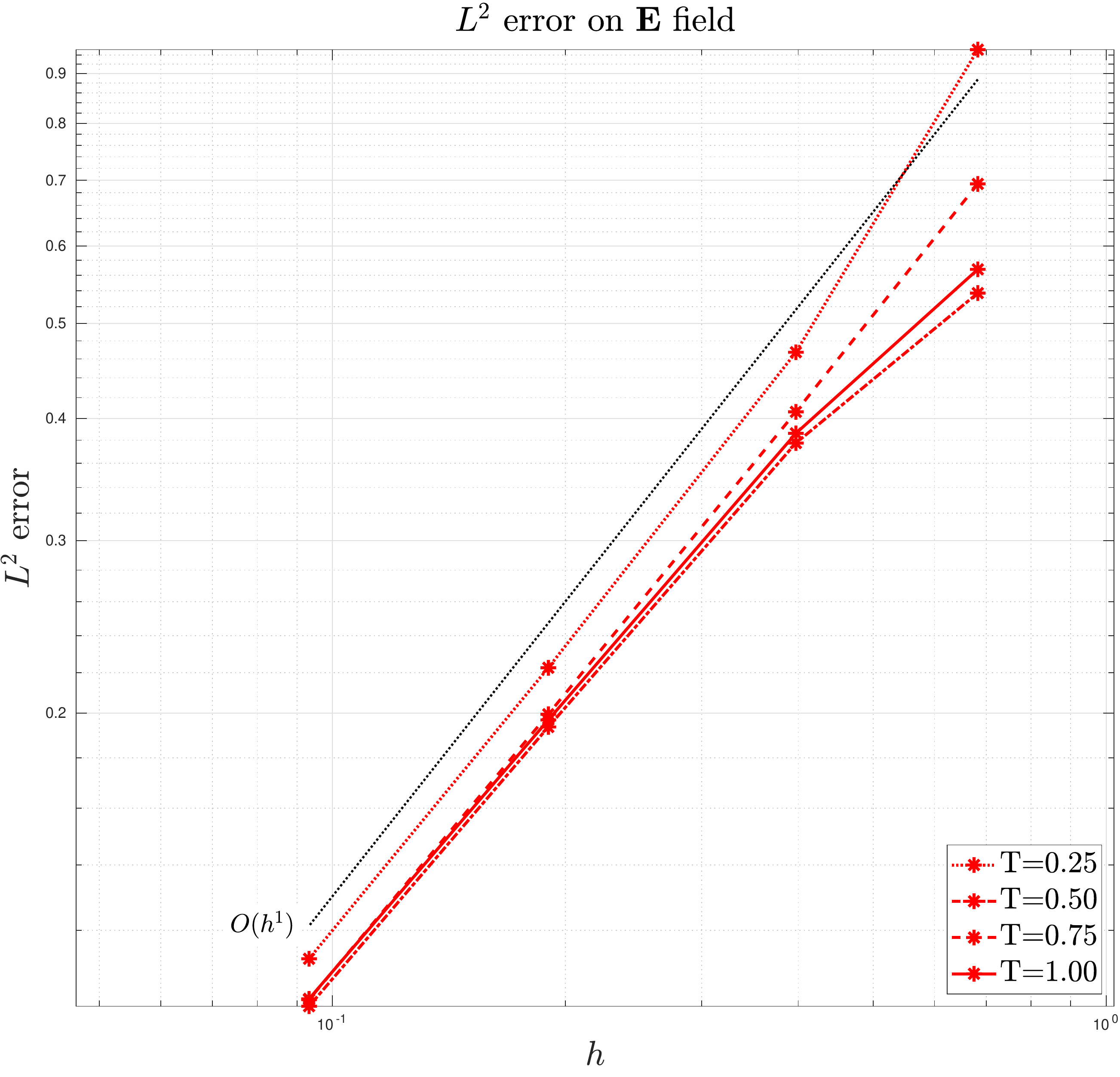} &
    \includegraphics[width=0.45\textwidth]{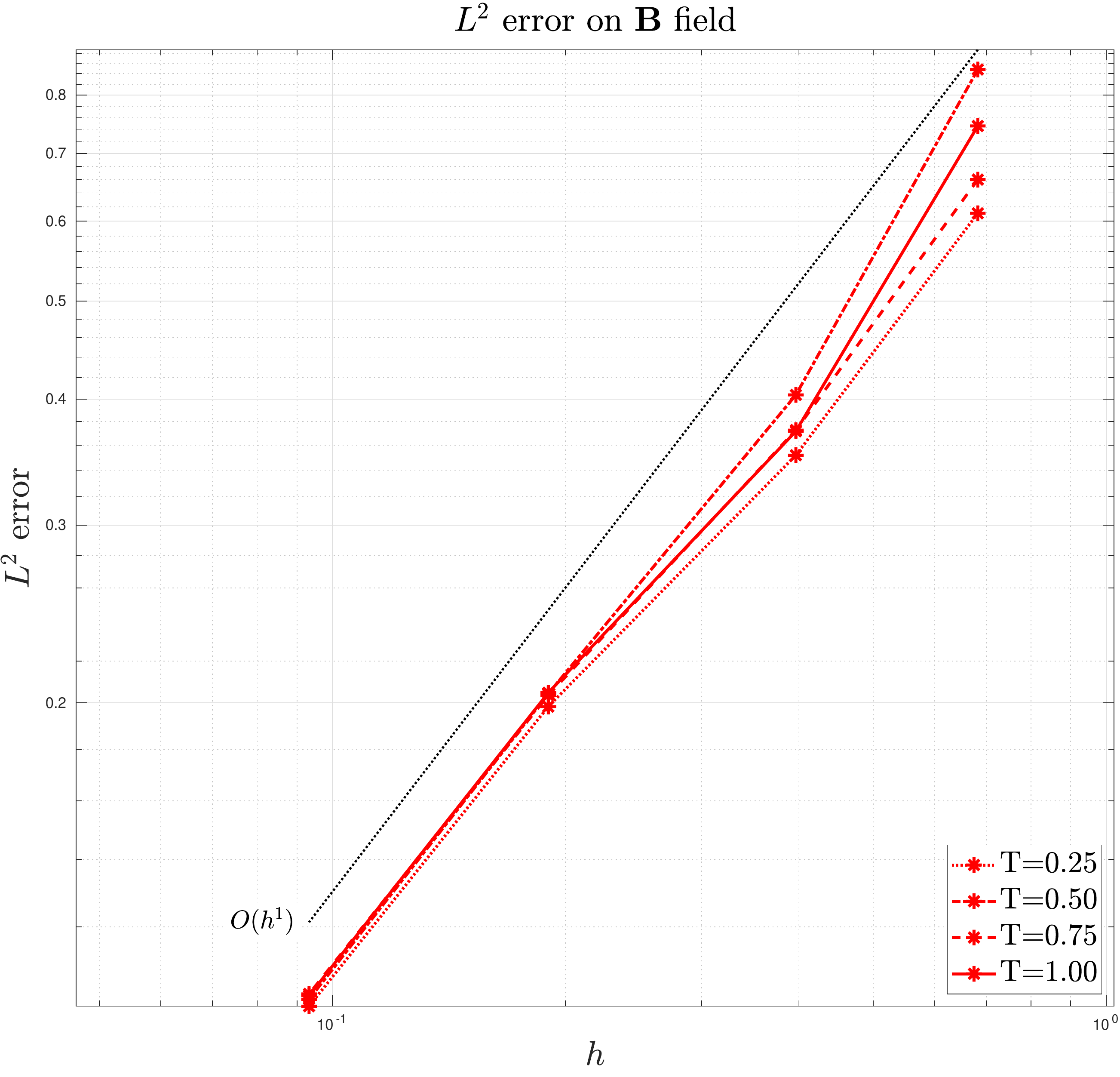} \\
  \end{tabular}
\caption{Test Case~2: error curves in the $L^2$-norm at times~$T=0.25, 0.50, 0.75$ and $1.00$ for the virtual element approximation $\Ebold_h$ (left panel) and $\Bbold_h$ (right panel) using the mesh family \texttt{rand} for simultaneous refinement of~$h$ and~$\tau$.}
\label{fig:convEBDifferentTime}
\end{figure}

As in Test Case 1,
we observe similar convergence behavior of the proposed VEM scheme on each mesh families so
in Tables~\ref{tab:voroEConvIE2} and~\ref{tab:voroEConvIB2},
we report the approximation errors measured in the $L^2$ norms only for \texttt{voro} meshes and we omit the results for the other two mesh families.

Despite the increased complexity due to variable coefficients, 
we observe the optimal convergence behavior of the error also in this example.
Indeed, each column shows the convergence with respect to the space discretization, 
each row shows the convergence with respect to the time discretization and the diagonal shows the convergence when we refine simultaneously in space and time.
As before, the error of the space discretization appears to dominate the error of the time discretization. Thus, the convergence in time along the rows, which should be proportional to $\tau$, is not clearly visible.

\begin{table}[!htb]
  \centering
\begin{tabular}{|r|c|c|c|c|c|c|c|}
    \hline
    \multicolumn{1}{|c|}{$h/\tau$}&1/8 &1/16 &1/32 &1/64 &1/128 &1/256 &1/512\\
    \hline
    \texttt{voro27}    & 8.64460e-01 $\:$& 6.85496e-01 $\:$& 5.61864e-01 $\:$& 5.03712e-01 $\:$& 4.81700e-01 $\:$& 4.73853e-01 $\:$& 4.70936e-01 $\:$\\ 
    \texttt{voro125}   & 8.55032e-01 $\:$& 6.30865e-01 $\:$& 4.51896e-01 $\:$& 3.55114e-01 $\:$& 3.16074e-01 $\:$& 3.02671e-01 $\:$& 2.98186e-01 $\:$\\ 
    \texttt{voro1000}  & 8.44007e-01 $\:$& 5.92062e-01 $\:$& 3.76326e-01 $\:$& 2.43624e-01 $\:$& 1.81720e-01 $\:$& 1.59270e-01 $\:$& 1.52316e-01 $\:$\\ 
    \texttt{voro8000}  & 8.40933e-01 $\:$& 5.81424e-01 $\:$& 3.54326e-01 $\:$& 2.05935e-01 $\:$& 1.27080e-01 $\:$& 9.33219e-02 $\:$& 8.18718e-02 $\:$\\ 
    \hline
\end{tabular}
\caption{Test Case~2: $L^2$-norms of the error at final time~$T=1$ for the virtual element approximation $\Ebold_h$ using mesh family~\texttt{voro} for various combinations of~$h$ and~$\tau$.}
\label{tab:voroEConvIE2}
\end{table}

\begin{table}[!htb]
  \centering
  \begin{tabular}{|r|c|c|c|c|c|c|c|}
    \hline
    \multicolumn{1}{|c|}{$h/\tau$}&1/8 &1/16 &1/32 &1/64 &1/128 &1/256 &1/512\\
    \hline
    \texttt{voro27}    & 5.92004e-01 $\:$& 5.69835e-01 $\:$& 5.43713e-01 $\:$& 5.34765e-01 $\:$& 5.37512e-01 $\:$& 5.42274e-01 $\:$& 5.45764e-01 $\:$\\ 
    \texttt{voro125}   & 4.68226e-01 $\:$& 4.28029e-01 $\:$& 3.69947e-01 $\:$& 3.23317e-01 $\:$& 3.02412e-01 $\:$& 2.95933e-01 $\:$& 2.94413e-01 $\:$\\ 
    \texttt{voro1000}  & 4.06835e-01 $\:$& 3.58563e-01 $\:$& 2.82887e-01 $\:$& 2.10580e-01 $\:$& 1.70641e-01 $\:$& 1.55564e-01 $\:$& 1.51081e-01 $\:$\\ 
    \texttt{voro8000}  & 3.90261e-01 $\:$& 3.38876e-01 $\:$& 2.55327e-01 $\:$& 1.69054e-01 $\:$& 1.13551e-01 $\:$& 8.82538e-02 $\:$& 7.96198e-02 $\:$\\
    \hline
  \end{tabular}
\caption{Test Case~2: $L^2$-norms of the error at final time~$T=1$ for the virtual element approximation $\Bbold_h$ using mesh family~\texttt{voro} for various combinations of $h$ and~$\tau$.}
\label{tab:voroEConvIB2}
\end{table}

Finally, Table~\ref{tab:voroDivBConvIE2} shows the values of the $L^2$-norm of $\div \Bbold_h$:
the VEM does preserve the solenoidal property of the magnetic induction, i.e., the discrete field~$\Bbold_h$ has a pointwise zero divergence up to machine precision.
If we compare the results in Tables~\ref{tab:voroDivBConvIE1} and Table~\ref{tab:voroDivBConvIE2}, then we note that the latters are smaller.
This is a further numerical evidence of the fact that the divergence-free property is affected by the  condition number of the resulting linear system.
Indeed, \texttt{voro} meshes are more shape-regular with respect to \texttt{rand} ones
so the condition numbers of matrices associated with them are smaller than 
those associated with \texttt{rand} meshes: 
the algebraic errors are smaller and we get a smaller divergence.

\begin{table}[!htb]
  \centering
  \begin{tabular}{|r|c|c|c|c|c|c|c|}
    \hline
    \multicolumn{1}{|c|}{$h/\tau$}&1/8 &1/16 &1/32 &1/64 &1/128 &1/256 &1/512\\
    \hline
    \texttt{voro27}    & 1.18487e-13 $\:$& 6.65328e-14 $\:$& 7.74812e-14 $\:$& 9.42188e-14 $\:$& 1.02437e-13 $\:$& 9.95126e-14 $\:$& 1.04040e-13 $\:$\\ 
    \texttt{voro125}   & 1.79095e-15 $\:$& 4.39464e-15 $\:$& 2.87788e-15 $\:$& 4.01614e-15 $\:$& 6.12453e-15 $\:$& 8.84695e-15 $\:$& 1.02443e-14 $\:$\\ 
    \texttt{voro1000}  & 1.98678e-14 $\:$& 3.86910e-14 $\:$& 3.81884e-14 $\:$& 3.34432e-14 $\:$& 3.23360e-14 $\:$& 2.54928e-14 $\:$& 2.88468e-14 $\:$\\ 
    \texttt{voro8000}  & 1.75878e-13 $\:$& 5.95562e-13 $\:$& 2.50332e-13 $\:$& 2.40717e-13 $\:$& 1.63504e-13 $\:$& 1.48763e-13 $\:$& 9.27908e-14 $\:$\\ 
    \hline
  \end{tabular}
\caption{Test Case~2: $L^2$ norm of $\div \Bbold_h$ at final time~$T=1$ using mesh family~\texttt{voro} for various combinations of $h$ and~$\tau$.}
\label{tab:voroDivBConvIE2}
\end{table}

\section{Conclusions}
\label{section:conclusions}

In this paper, we have considered a low order virtual element approximation of Maxwell's equations
based on a De Rahm sequence.
After developing some interpolation and stability properties of edge and face spaces,
we showed optimal a priori estimates for both the semi- and the fully- discrete schemes
and corroborated the theoretical predictions with numerical experiments.
Future works may cover the approximation of corner singularities and the virtual element approximation of MHD problems.
The extension to high order methods requires high order interpolation estimates and stability properties of edge and face VEM spaces, which is currently a work in progress.

\section*{Acknowledgments}
L.~B. da V. and F.~D. are partially supported by the European Research Council through the H2020 Consolidator Grant (Grant No. 681162) CAVE - Challenges and Advancements in Virtual Elements. 
L.~B. da V. is also partially supported by the MIUR through the PRIN grant n. 201744KLJL.
G.~M. has partially been supported by the ERC Project CHANGE, which has received funding from the European Research Council (ERC) under the European Union Horizon 2020 research and innovation program (grant agreement no.~694515).
L.~M. acknowledges support from the Austrian Science Fund (FWF) project P33477.

We further wish to thank Martin Costabel for an advice regarding a regularity result.


\bibliography{bibliogr}

\begin{thebibliography}{10}

\bibitem{adamsfournier}
R.~A. Adams and J.~J.~F. Fournier.
\newblock {\em Sobolev {S}paces}, volume 140.
\newblock Academic {P}ress, 2003.

\bibitem{equivalentprojectorsforVEM}
B.~Ahmad, A.~Alsaedi, F.~Brezzi, L.D. Marini, and A.~Russo.
\newblock {E}quivalent projectors for virtual element methods.
\newblock {\em Comput. Math. Appl.}, 66(3):376--391, 2013.

\bibitem{pardiso}
Christie Alappat, Achim Basermann, Alan~R. Bishop, Holger Fehske, Georg Hager,
  Olaf Schenk, Jonas Thies, and Gerhard Wellein.
\newblock A recursive algebraic coloring technique for hardware-efficient
  symmetric sparse matrix-vector multiplication.
\newblock {\em ACM Trans. Parallel Comput.}, 7(3), June 2020.

\bibitem{alvarez2020virtual}
S.~N. Alvarez, V.~A. Bokil, V.~Gyrya, and G.~Manzini.
\newblock The virtual element method for resistive magnetohydrodynamics.
\newblock {\em Comput. Methods Appl. Mech. Engrg.}, 381:113815, 2021.

\bibitem{amrouche1998vector}
C.~Amrouche, C.~Bernardi, M.~Dauge, and V.~Girault.
\newblock Vector potentials in three-dimensional non-smooth domains.
\newblock {\em Math. Methods Appl. Sci.}, 21(9):823--864, 1998.

\bibitem{Assous-Maxwell}
F.~Assous, P.~Degond, E.~Heintze, P.-A. Raviart, and J.~Segre.
\newblock On a finite-element method for solving the three-dimensional
  {M}axwell equations.
\newblock {\em J. Comput. Phys.}, 109(2):222--237, 1993.

\bibitem{face-edge-VEM-interpolation-low}
L.~Beir\~ao~da Veiga and L.~Mascotto.
\newblock Interpolation and stability properties of low order face and edge
  virtual element spaces.
\newblock https://arxiv.org/abs/2011.12834, 2020.

\bibitem{VEMvolley}
L.~Beir{\~a}o~da Veiga, F.~Brezzi, A.~Cangiani, G.~Manzini, L.D. Marini, and
  A.~Russo.
\newblock Basic principles of virtual element methods.
\newblock {\em Math. Models Methods Appl. Sci.}, 23(01):199--214, 2013.

\bibitem{da2017virtual}
L.~Beir{\~a}o~da Veiga, F.~Brezzi, F.~Dassi, L.~D. Marini, and A.~Russo.
\newblock Virtual element approximation of {2D} magnetostatic problems.
\newblock {\em Comput. Methods Appl. Mech. Engrg.}, 327:173--195, 2017.

\bibitem{da2018family}
L.~Beir{\~a}o~da Veiga, F.~Brezzi, F.~Dassi, L.~D. Marini, and A.~Russo.
\newblock A family of three-dimensional virtual elements with applications to
  magnetostatics.
\newblock {\em SIAM J. Numer. Anal.}, 56(5):2940--2962, 2018.

\bibitem{da2018lowest}
L.~Beir{\~a}o~da Veiga, F.~Brezzi, F.~Dassi, L.~D. Marini, and A.~Russo.
\newblock Lowest order virtual element approximation of magnetostatic problems.
\newblock {\em Comput. Methods Appl. Mech. Engrg.}, 332:343--362, 2018.

\bibitem{HdivHcurlVEM}
L.~Beir{\~a}o Da~Veiga, F.~Brezzi, L.~D. Marini, and A.~Russo.
\newblock {H}(div) and {H}(curl)-conforming virtual element methods.
\newblock {\em Numer. Math}, 133(2):303--332, 2016.

\bibitem{BBMR_generalsecondorder}
L.~Beirao~da Veiga, F.~Brezzi, L.~D. Marini, and A.~Russo.
\newblock Mixed virtual element methods for general second order elliptic
  problems on polygonal meshes.
\newblock {\em ESAIM Math. Model. Numer. Anal.}, 50(3):727--747, 2016.

\bibitem{VEM3Dbasic}
L.~Beir{\~a}o~da Veiga, F.~Dassi, and A.~Russo.
\newblock High-order virtual element method on polyhedral meshes.
\newblock {\em Comput. Math. Appl.}, 74(5):1110--1122, 2017.

\bibitem{beiraolovadinarusso_stabilityVEM}
L.~Beir{\~a}o~da Veiga, C.~Lovadina, and A.~Russo.
\newblock Stability analysis for the virtual element method.
\newblock {\em Math. Models Methods Appl. Sci.}, 27(13):2557--2594, 2017.

\bibitem{Bermudez-book}
A.~Berm{\'u}dez~de Castro, D.~G{\'o}mez, and P.~Salgado.
\newblock {\em Mathematical models and numerical simulation in
  electromagnetism}, volume~74.
\newblock Springer, 2014.

\bibitem{BrezziFortin}
D.~Boffi, F.~Brezzi, and M.~Fortin.
\newblock {\em Mixed {F}inite {E}lement {M}ethods and {A}pplications},
  volume~44.
\newblock Springer Series in Computational Mathematics, 2013.

\bibitem{brennerVEMsmall}
S.~C. Brenner and L.-Y. Sung.
\newblock Virtual element methods on meshes with small edges or faces.
\newblock {\em Math. Models Methods Appl. Sci.}, 268(07):1291--1336, 2018.

\bibitem{Brezzi-Falk-Marini}
F.~Brezzi, R.~S. Falk, and L.D. Marini.
\newblock Basic principles of mixed virtual element methods.
\newblock {\em Math. Mod. Num. Anal.}, 48(4):1227--1240, 2014.

\bibitem{chave2020three}
F.~Chave, D.~A. Di~Pietro, and S.~Lemaire.
\newblock A three-dimensional hybrid high-order method for magnetostatics.
\newblock In {\em International Conference on Finite Volumes for Complex
  Applications}, pages 255--263. Springer, 2020.

\bibitem{Ciarlet-Zou}
P.~Ciarlet, Jr. and J.~Zou.
\newblock Fully discrete finite element approaches for time-dependent
  {M}axwell's equations.
\newblock {\em Numer. Math.}, 82(2):193--219, 1999.

\bibitem{Coccioli-Itoh-Pelosi-Silvester:1996}
R.~Coccioli, T.~Itoh, G.~Pelosi, and P.~P. Silvester.
\newblock Finite-element methods in microwaves: {A} selected bibliography.
\newblock {\em Antennas and Propagation Newsletter, IEEE Professional Group
  on}, 38, Dec. 1996.

\bibitem{Dassi-DiBarba-Russo}
F.~Dassi, P.~Di~Barba, and A.~Russo.
\newblock Virtual element method and permanent magnet simulations: potential
  and mixed formulations.
\newblock {\em IET Science, Measurement \& Technology}, 14(10):1098--1104,
  2021.

\bibitem{dauge2006elliptic}
M.~Dauge.
\newblock {\em Elliptic boundary value problems on corner domains: smoothness
  and asymptotics of solutions}, volume 1341.
\newblock Springer, 2006.

\bibitem{DeRahm-skeletal}
D.~A. Di~Pietro, J.~Droniou, and F.~Rapetti.
\newblock Fully discrete polynomial de {R}ham sequences of arbitrary degree on
  polygons and polyhedra.
\newblock {\em Math. Models Methods Appl. Sci.}, 30(9):1809--1855, 2020.

\bibitem{Euler-Schuhmann-Weiland:2006}
T.~Euler, R.~Schuhmann, and T.~Weiland.
\newblock Polygonal finite elements.
\newblock {\em IEEE Trans. Magnetics}, 42, 2006.

\bibitem{evansPDE}
L.~C. Evans.
\newblock {\em Partial {D}ifferential {E}quations}.
\newblock American {M}athematical {S}ociety, 2010.

\bibitem{Greenwood-Jin:1999}
A.~D. Greenwood and J.-M. Jin.
\newblock Finite-element analysis of complex axisymmetric radiating structures.
\newblock {\em IEEE Trans. Antennas Propagation}, 47(8):1260--1266, 1999.

\bibitem{Jin:2014}
J.-M. Jin.
\newblock {\em The finite element method in electromagnetics}.
\newblock John Wiley \& Sons, third edition edition, 2014.

\bibitem{Jin-Riley:2009}
J.-M. Jin and D.~J. Riley.
\newblock {\em Finite {E}lement {A}nalysis of {A}ntennas and {A}rrays}.
\newblock John Wiley \& Sons :, IEEE Press, 2009.

\bibitem{Khebir-DAngelo-Joseph:1993}
A.~Khebir, J.~{D'Angelo}, and J.~Joseph.
\newblock A new finite element formulation for {RF} scattering by complex
  bodies of revolution.
\newblock {\em IEEE Trans. Antennas Propagation}, 41(5):534--541, 1993.

\bibitem{Lee-Wilkins-Mitra:1993}
J.~F. Lee, G.M. Wilkins, and R.~Mitra.
\newblock Finite-element analysis of axisymmetric cavity resonator using a
  hybrid edge element technique.
\newblock {\em IEEE Trans. Microwave Theory Techniques}, 41:1981--1987, 1993.

\bibitem{MFD-3D-Maxwell}
K.~Lipnikov, G.~Manzini, F.~Brezzi, and A.~Buffa.
\newblock The mimetic finite difference method for the 3{D} magnetostatic field
  problems on polyhedral meshes.
\newblock {\em J. Comput. Phys.}, 230(2):305--328, 2011.

\bibitem{Monk-Makridakis}
Ch.~G. Makridakis and P.~Monk.
\newblock Time-discrete finite element schemes for {M}axwell's equations.
\newblock {\em RAIRO Mod\'{e}l. Math. Anal. Num\'{e}r.}, 29(2):171--197, 1995.

\bibitem{MedgyesiMitschang-Putnam:1984}
L.~{Medgyesi-Mitschang} and J.~Putnam.
\newblock Electromagnetic scattering from axially inhomogeneous bodies of
  revolution.
\newblock {\em IEEE Trans. Antennas Propagation}, 32(8):797--806, 1984.

\bibitem{monk2003finite}
P.~Monk.
\newblock {\em Finite {E}lement {M}ethods for {M}axwell's {E}quations}.
\newblock Oxford University Press, 2003.

\bibitem{Monk1991}
P.~B. Monk.
\newblock A mixed method for approximating {M}axwell's equations.
\newblock {\em SIAM J. Numer. Anal.}, 28(6):1610--1634, 1991.

\bibitem{VEMchileans}
D.~Mora, G.~Rivera, and R.~Rodr{\'i}guez.
\newblock A virtual element method for the {S}teklov eigenvalue problem.
\newblock {\em Math. Models Methods Appl. Sci.}, 25(08):1421--1445, 2015.

\bibitem{Nedelec1980}
J.-C. N\'{e}d\'{e}lec.
\newblock Mixed finite elements in {${\bf R}^{3}$}.
\newblock {\em Numer. Math.}, 35(3):315--341, 1980.

\bibitem{Rui-Hu-Liu:2010}
X.~Rui, J.~Hu, and Q.~H. Liu.
\newblock Higher order finite element method for inhomogeneous axisymmetric
  resonators.
\newblock {\em Progress In Electromagnetics Research B}, 21:189--201, 2010.

\bibitem{Teixeira-Bergmann:1997-journal}
F.~L. Teixeira and J.~R. Bergmann.
\newblock B-spline basis functions for moment-method analysis of axisymmetric
  reflector antennas.
\newblock {\em Microwave and Optical Technology Letters}, 14, 1997.

\bibitem{Teixeira-Bergmann:1997-proc}
F.~L. Teixeira and J.~R. Bergmann.
\newblock Moment-method analysis of circularly symmetric reflectors using
  bandlimited basis functions.
\newblock {\em IEE Proc. Microw. Antennas Prop.}, 144(3):179--183, 1997.

\bibitem{Tierens-DeZutter:2011}
W.~Tierens and D.~De~Zutter.
\newblock {BOR-FDTD} subgridding based on finite element principles.
\newblock {\em J. Comput. Phys.}, 230, 2011.

\bibitem{Wilkins-Lee-Mittra:1991}
G.M. Wilkins, J.-F. Lee, and R.~Mittra.
\newblock Numerical modeling of axisymmetric coaxial waveguide discontinuities.
\newblock {\em IEEE Trans. Microwave Theory Techniques}, 39, 1991.

\bibitem{zhao2004analysis}
J.~Zhao.
\newblock Analysis of finite element approximation for time-dependent {M}axwell
  problems.
\newblock {\em Math. Comp.}, 73(247):1089--1105, 2004.

\end{thebibliography}
\bibliographystyle{plain}


\appendix

\end{document}